\documentclass[10pt]{amsart}

\usepackage[english]{babel}

\usepackage{amssymb}
\usepackage[leqno]{amsmath}
\usepackage{mathrsfs}
\usepackage{stmaryrd}
\usepackage{chemarrow}
\usepackage[norelsize]{algorithm2e}
\usepackage{enumerate}
\usepackage{graphicx}
\usepackage[all]{xy}
\usepackage{tikz,tikz-cd}
\usetikzlibrary{arrows}
\usepackage{multirow}
\usepackage[colorinlistoftodos]{todonotes}
\usepackage[colorlinks=true, allcolors=blue]{hyperref}

\newtheorem{theorem}{Theorem}[section]
\newtheorem{lemma}[theorem]{Lemma}

\newtheorem{remark}[theorem]{Remark}

\newcommand{\dd}{\,{\rm d}}
\newcommand{\bs}{\boldsymbol}

\DeclareMathOperator*{\img}{img}
\newcommand{\sign}{\operatorname{sign}}
\newcommand{\curl}{\operatorname{curl}}
\renewcommand{\div}{\operatorname{div}}
\newcommand{\grad}{\operatorname{grad}}
\DeclareMathOperator*{\tr}{tr}
\DeclareMathOperator*{\rot}{rot}

\newcommand{\sym}{\operatorname{sym}}
\newcommand{\skw}{\operatorname{skw}}

\newcommand{\mskw}{\operatorname{mskw}}
\newcommand{\vskw}{\operatorname{vskw}}

\newcommand{\defm}{\operatorname{def}}

\newcommand{\inc}{\operatorname{inc}}

\usepackage[hyperpageref]{backref}

\begin{document}
\title{A finite element elasticity complex in three dimensions}
 \author{Long Chen}%
 \address{Department of Mathematics, University of California at Irvine, Irvine, CA 92697, USA}%
 \email{chenlong@math.uci.edu}%
 \author{Xuehai Huang}%
 \address{School of Mathematics, Shanghai University of Finance and Economics, Shanghai 200433, China}%
 \email{huang.xuehai@sufe.edu.cn}%

 \thanks{The first author was supported by NSF DMS-1913080 and DMS-2012465.}
 \thanks{The second author was supported by the National Natural Science Foundation of China Project 11771338, the Natural Science Foundation of Shanghai 21ZR1480500  and the Fundamental Research Funds for the Central
Universities 2019110066.}


\begin{abstract}
A finite element elasticity complex on tetrahedral meshes is devised. The $H^1$ conforming finite element is the smooth finite element developed by Neilan for the velocity field in a discrete Stokes complex. The symmetric div-conforming finite element is the Hu-Zhang element for stress tensors. The construction of an $H(\inc)$-conforming finite element for symmetric tensors is the main focus of this paper. 
The key tools of the construction are the decomposition of polynomial tensor spaces and the characterization of the trace of the $\inc$ operator. 
The polynomial elasticity complex and Koszul elasticity complex are created to derive the decomposition of polynomial tensor spaces. 
The trace of the $\inc$ operator is induced from a Green's identity. Trace complexes and bubble complexes are also derived to facilitate the construction. Our construction appears to be the first $H(\inc)$-conforming finite elements on tetrahedral meshes without further splits. 
\end{abstract}

\maketitle


\section{Introduction}
A Hilbert complex is a sequence of Hilbert spaces connected by a sequence of linear operators satisfying the property: the composition of two consecutive operators is vanished. Let $\Omega$ be a bounded domain in $\mathbb R^3$. The elasticity complex
\begin{equation}\label{eq:elasticitycomplex}
\boldsymbol{RM}\stackrel{\subset}{\longrightarrow} \boldsymbol  H^{1} (\Omega;\mathbb R^3) \stackrel{\defm}{\longrightarrow} \boldsymbol  H(\mathrm{inc}, \Omega;\mathbb{S}) \stackrel{\text { inc }}{\longrightarrow} \boldsymbol  H(\operatorname{div}, \Omega; \mathbb{S}) \stackrel{\text { div }}{\longrightarrow} \boldsymbol  L^{2} (\Omega; \mathbb R^3)\rightarrow 0
\end{equation}
plays an important role in both theoretical and numerical study of linear elasticity, where $\bs{RM}$ is the space of the linearized rigid body motion, $\defm$ is the symmetric gradient operator, $\boldsymbol   H(\mathrm{inc}, \Omega;\mathbb{S}) $ is the space of symmetric tensor $\boldsymbol  \tau$ s.t. $\inc \boldsymbol  \tau := - \curl (\curl \boldsymbol  \tau)^{\intercal}\in \boldsymbol  L^2(\Omega; \mathbb M)$, and  $\boldsymbol  H(\operatorname{div}, \Omega; \mathbb{S})$ is the space for the symmetric stress tensor $\boldsymbol  \sigma$ with $\div\boldsymbol  \sigma\in \boldsymbol  L^2(\Omega; \mathbb R^3)$. 
We shall present a finite element elasticity complex 
\begin{equation}\label{eq:introelascomplex3dfem}
\boldsymbol{RM}\stackrel{\subset}{\longrightarrow} \boldsymbol  V_h \stackrel{\defm}{\longrightarrow} \bs\Sigma_h^{\rm inc} \stackrel{\text { inc }}{\longrightarrow} \bs\Sigma_h^{\rm div}  \stackrel{\text { div }}{\longrightarrow} \mathcal Q_h\rightarrow 0
\end{equation}
on a tetrahedral mesh $\mathcal T_h$. In the complex \eqref{eq:introelascomplex3dfem}, the $H^1$-conforming finite element is the smooth finite element $\boldsymbol  V_h$ developed by Neilan for the velocity field in a finite element Stokes complex~\cite{Neilan2015}. The $\boldsymbol  H(\div;\mathbb S)$-conforming finite element $\bs\Sigma_h^{\rm div}$ is the Hu-Zhang element for the symmetric stress tensor~\cite{Hu2015,HuZhang2015}. The space $\mathcal Q_h$ for $\boldsymbol  L^2(\Omega;\mathbb R^3)$ is simply the discontinuous piecewise polynomial space. The missing component is an $\boldsymbol  H(\inc; \mathbb S)$-conforming finite element $\bs\Sigma_h^{\rm inc}$ which is the focus of this work. 

In the engineering application, the most important component in the complex \eqref{eq:introelascomplex3dfem} is the finite element $\bs\Sigma_h^{\rm div}$ for the stress tensor. Construction of finite element for stress tensors can be benefit from the structure of the complex. For example, the bubble polynomial elasticity complex is build and used in~\cite{ArnoldAwanouWinther2008} to construct a finite element for symmetric stress tensors. Here the bubble polynomial spaces are referred to polynomials with vanished traces on the boundary of each element. In~\cite{HuZhang2015}, a precise characterization of $\boldsymbol  H(\div; \mathbb S)$ bubble polynomial space is given which leads to a stable $\mathbb P_{k}(\mathbb S) -\mathbb P_{k-1}(\mathbb R^d)$ stress-displacement finite element pair for $k\geq d+1, d=2,3$. Identification of its preceding space $\bs\Sigma_h^{\rm inc}$ will be helpful for the design of fast solvers and a posterior error analysis~\cite{Chen;Hu;Huang;Man:2018Residual-based} for the mixed formulation of linear elasticity problems. It may also find application in other fields such as continuum modeling of defects~\cite{Amstutz;Van-Goethem:2019incompatibility} and relativity~\cite{Christiansen:2011linearization}.

Elasticity complex \eqref{eq:elasticitycomplex} and many more complexes can be derived from composition of de Rham complexes in the so-called Bernstein-Gelfand-Gelfand (BGG) construction~\cite{Arnold;Hu:2020Complexes}. Finite element complexes for the de Rham complex are well understood and can be derived systematically in the framework of Finite Element Exterior Calculus~\cite{Arnold;Falk;Winther:2010Finite,Arnold;Falk;Winther:2006Finite}. It is natural to ask if a finite element elasticity complex can be derived by the BGG construction. The key in the BGG construction is the existence of smooth finite element de Rham complexes. With nodal finite element de Rham complexes, a two dimensional finite element elasticity complex has been constructed in~\cite{Christiansen;Hu;Hu:2018finite} which generalizes the first finite element elasticity complex of Arnold and Winther~\cite{Arnold;Winther:2002finite}. 

In three dimensions, however, smooth discrete de Rham complexes are not easy due to the super-smoothness of multivariate splines~\cite{Floater;Hu:2020characterization}. To relax the super-smoothness, the element can be further split so that inside one element the shape function is not $\mathcal C^{\infty}$ smooth. Such approach leads to the so-called macro elements. In particular, a two dimensional elasticity strain complex has been constructed on the Clough-Tocher split of a triangle~\cite{christiansen2019finite}, and more recently a finite element elasticity complex has been constructed on the Alfeld split of a tetrahedron~\cite{Christiansen;Gopalakrishnan;Guzman;Hu:2020discrete} based on the smooth finite element de Rham complex~\cite{Fu;Guzman;Neilan:2020smooth} on such split. 

We shall construct a finite element elasticity complex on a tetrahedral mesh without further split. Let $K$ be a polyhedron. We first give a polynomial elasticity complex and a Koszul type complex, which can be summarized as one diagram below:
\begin{equation*}
\resizebox{.995\hsize}{!}{$
\xymatrix{
\boldsymbol{RM}\ar@<0.4ex>[r]^-{\subset} & \mathbb P_{k+1}(K;\mathbb R^3)\ar@<0.4ex>[r]^-{\defm}\ar@<0.4ex>[l]^-{\pi_{RM}} & \mathbb P_{k}(K;\mathbb S)\ar@<0.4ex>[r]^-{\inc}\ar@<0.4ex>[l]^-{ \boldsymbol\tau\cdot\boldsymbol x}  & \mathbb P_{k-2}(K;\mathbb S) \ar@<0.4ex>[r]^-{\div}\ar@<0.4ex>[l]^-{\boldsymbol x\times\boldsymbol\tau \times \boldsymbol x} & \mathbb P_{k-3}(K;\mathbb R^3)  \ar@<0.4ex>[r]^-{} \ar@<0.4ex>[l]^-{\sym(\boldsymbol v\boldsymbol x^{\intercal})}
& 0 \ar@<0.4ex>[l]^-{\supset} }.
$}
\end{equation*}
Several decompositions of polynomial tensor spaces, especially for $\mathbb P_k(K; \mathbb S)$, can be obtained consequently. We then study trace operators for the $\inc$ operator since the traces on face and edges are crucial to ensure the $H(\inc)$-conformity.  To do so, we use more symmetric notation $\inc \boldsymbol  \tau = \nabla \times \bs\tau \times \nabla$ and derive the following symmetric form Green's identity:
 \begin{align*}
   (\nabla \times \bs\sigma \times \nabla, \boldsymbol  \tau)_K - (\boldsymbol  \sigma,\nabla \times \bs\tau \times \nabla)_K &= ({\rm tr}_1(\boldsymbol  \sigma), {\rm tr}_2(\boldsymbol  \tau))_{\partial K} - ( {\rm tr}_2(\boldsymbol  \sigma), {\rm tr}_1(\boldsymbol  \tau))_{\partial K} \notag\\
   &\quad + \sum_{F\in \mathcal F(K)}\sum_{e\in \partial F}(\boldsymbol  n \cdot\boldsymbol  \sigma \times \boldsymbol  n, \boldsymbol  t_{F,e}\cdot \boldsymbol  \tau)_{e} \\
   &\quad - \sum_{F\in \mathcal F(K)}\sum_{e\in \partial F}(\boldsymbol  t_{F,e}\cdot \boldsymbol  \sigma, \boldsymbol  n \cdot\boldsymbol  \tau \times \boldsymbol  n)_{e},
\end{align*}
where, with $\Pi_F$ denoting the projection operator to face $F$, 
\begin{align*}
{\rm tr}_1(\boldsymbol  \tau) &:=  \boldsymbol  n\times\boldsymbol  \tau \times \boldsymbol  n, \\
   {\rm tr}_2(\boldsymbol  \tau) &:= \Pi_F(\boldsymbol  \tau\times\nabla)\times\boldsymbol  n + \nabla_F(\boldsymbol  n\cdot \boldsymbol  \tau \ \Pi_F).
\end{align*}
We show ${\rm tr}_1(\boldsymbol  \tau)\in \boldsymbol  H(\div_F\div_F, F;\mathbb S)$ and ${\rm tr}_2(\boldsymbol  \tau) \in  \boldsymbol  H(\rot_F, F;\mathbb S)$, and reveal boundary complexes induced by trace operators; see Section \ref{sec:tracecomplex} for details. Then the edge traces of the face traces ${\rm tr}_1(\boldsymbol  \tau)$ and ${\rm tr}_2(\boldsymbol  \tau)$ imply the continuity of $\boldsymbol  \tau|_e$ and $(\nabla \times \boldsymbol  \tau)\cdot \boldsymbol  t_e$. Further edge degree of freedoms will be derived from the requirement $\inc \boldsymbol  \tau$ is in Hu-Zhang finite element space. The face degree of freedom will be based on the decomposition of polynomial tensors of $ \boldsymbol  H(\div_F\div_F, F;\mathbb S)$ and $ \boldsymbol  H(\rot_F, F;\mathbb S)$. The volume degree of freedom is from the decomposition of $\mathbb P_{k}(K;\mathbb S)$. 

Recently there has been a lot progress in the construction of finite elements for tensors~\cite{Chen;Hu;Huang:2018Multigrid,Christiansen;Hu;Hu:2018finite,christiansen2019finite,ChenHuang2020,Chen;Huang:2020Discrete,Chen;Huang:2020Finite,HuLiang2021,Christiansen;Gopalakrishnan;Guzman;Hu:2020discrete}. Our construction appears to be the first $H(\inc)$-conforming finite elements for symmetric tensors on tetrahedral meshes without further splits. Our finite element spaces are constructed for tetrahedrons but some results, e.g., traces and Green's formulae etc, hold for general polyhedrons.  Our approach of constructing finite element for tensors, through decomposition of polynomial space and characterization of trace operators, seems simpler and more straightforward than the BGG construction through smooth finite element de Rham complexes. 

\subsection*{Notation on meshes}
Let $\{\mathcal {T}_h\}_{h>0}$ be a regular family of polyhedral meshes
of $\Omega$. For each element $K\in\mathcal{T}_h$, denote by $\boldsymbol{n}_K $ the unit outward normal vector to $\partial K$. In most places, it will be abbreviated as $\boldsymbol{n}$ for simplicity. Denote by $\mathcal{F}(K)$, $\mathcal{E}(K)$ and $\mathcal{V}(K)$ the set of all faces, edges and vertices of $K$, respectively. 
Similarly let $\mathcal{E}(F)$ be the set of all edges of face $F$. For $F\in \mathcal F(K)$, its orientation is given by the outwards normal direction $\boldsymbol  n_{\partial K}$ which also induces a consistent orientation of edge $e\in \mathcal E(F)$. Namely the edge vectors $\boldsymbol  t_{F,e}$ and outwards normal vector $\boldsymbol  n_{\partial K}$ follows the right hand rule. Then define $\boldsymbol  n_{F,e} = \boldsymbol  t_{F,e}\times \boldsymbol  n_{\partial K}$ as the outwards normal vector of $e$ on the face $F$. 

Let $\mathcal{F}_h$, $\mathcal{E}_h$, and $\mathcal{V}_h$ be the union of all faces, all edges, vertices and of the partition $\mathcal {T}_h$, respectively.
For any $F\in\mathcal{F}_h$,
fix a unit normal vector $\boldsymbol{n}_F$ and two unit tangent vectors $\boldsymbol{t}_{F,1}$ and $\boldsymbol{t}_{F,2}$, which will be abbreviated as $\boldsymbol{t}_{1}$ and $\boldsymbol{t}_{2}$ without causing any confusions.
For any $e\in\mathcal{E}_h$,
fix a unit tangent vector $\boldsymbol{t}_e$ and two unit normal vectors $\boldsymbol{n}_{e,1}$ and $\boldsymbol{n}_{e,2}$, which will be abbreviated as $\boldsymbol{n}_{1}$ and $\boldsymbol{n}_{2}$ without causing any confusions. We emphasize that $\boldsymbol  n_F, \boldsymbol  t_e, \boldsymbol  n_{e,1}$, and $\boldsymbol  n_{e,2}$ are globally defined not depending on the elements. 

The rest of this paper is organized as follows. In Section 2, we present a notation system on the vector and tensor operations. We construct polynomial complexes and derive decompositions of polynomial tensors spaces related to the elasticity complex in Section 3.
In Section 4, we discuss traces for the $\inc$ operator based on the Green's identity, and present corresponding trace complexes and bubble complexes. In Section 5, we construct an $H(\inc)$-conforming finite element and a finite element elasticity complex in three dimensions.

\section{Vector and tensor operations}
One complication on the construction of finite elements for tensors is the notation system for tensor operations. We shall adapt the notation system used in the solid mechanic~\cite{Kelly:Mechanics}. In particular, we separate the row and column operations to the right and left sides of the matrix, respectively. 
\subsection{Tensor calculus}
Define the dot product and the cross product from the left 
$$\boldsymbol  b\cdot \boldsymbol  A, \quad \boldsymbol  b\times \boldsymbol  A,$$
which is applied column-wisely to the matrix $A$. When the vector is on the right of the matrix
$$
\boldsymbol  A\cdot \boldsymbol  b, \quad \boldsymbol  A\times \boldsymbol  b,
$$
the operation is defined row-wisely. Here for the clean of notation, when the vector $\boldsymbol  b$ is on the right, it is treated as a row-vector $\boldsymbol  b^{\intercal}$ while when on the left, it is a column vector.

The ordering of performing the row and column products does not matter which leads to the associative rule of the triple products
$$
\boldsymbol  b\times \boldsymbol  A\times \boldsymbol  c := (\boldsymbol  b\times \boldsymbol  A)\times \boldsymbol  c = \boldsymbol  b\times (\boldsymbol  A\times \boldsymbol  c).
$$
Similar rules hold for $\boldsymbol  b\cdot \boldsymbol  A\cdot \boldsymbol  c$ and $\boldsymbol  b\cdot \boldsymbol  A\times \boldsymbol  c$ and thus parentheses can be safely skipped. 
Another benefit is the transpose of products. For the transpose of product of two objects, we take transpose of each one, switch their order, and add a negative sign if it is the cross product. 

For two column vectors $\boldsymbol  u, \boldsymbol  v$, the tensor product $\boldsymbol  u\otimes \boldsymbol  v := \boldsymbol  u\boldsymbol  v^{\intercal}$ is a matrix which is also known as the dyadic product $\boldsymbol  u\boldsymbol  v: = \boldsymbol  u\boldsymbol  v^{\intercal}$ with more clean notation (one $^{\intercal}$ is skipped). The row-wise product and column-wise product with another vector will be applied to the neighbor vector
\begin{align*}
\boldsymbol  x\cdot (\boldsymbol  u\boldsymbol  v) = (\boldsymbol  x\cdot \boldsymbol  u) \boldsymbol  v^{\intercal}, \quad (\boldsymbol  u\boldsymbol  v)\cdot \boldsymbol  x = \boldsymbol  u (\boldsymbol  v\cdot \boldsymbol  x), \\
\boldsymbol  x\times (\boldsymbol  u\boldsymbol  v) = (\boldsymbol  x\times \boldsymbol  u) \boldsymbol  v, \quad (\boldsymbol  u\boldsymbol  v)\times \boldsymbol  x = \boldsymbol  u (\boldsymbol  v\times \boldsymbol  x).
\end{align*}

We treat Hamilton operator $\nabla = (\partial_1, \partial_2, \partial_3)^{\intercal}$ as a column vector. 
For a vector function $\boldsymbol  u = (u_1, u_2, u_3)^{\intercal}$, $\curl \boldsymbol  u= \nabla \times \boldsymbol  u$, and $\div \boldsymbol  u = \nabla \cdot \boldsymbol  u$ are standard differential operations. Define $\nabla \boldsymbol  u = \nabla \boldsymbol  u^{\intercal} = (\partial_i u_j)$ which can be understood as the dyadic product of Hamilton operator $\nabla$ and column vector $\boldsymbol  u$.

Apply these matrix-vector operations to the Hamilton operator $\nabla$, we get column-wise differentiation $\nabla \cdot \boldsymbol  A, \nabla \times \boldsymbol  A,$
and row-wise differentiation
$\boldsymbol  A\cdot \nabla, \boldsymbol  A\times \nabla.$ Conventionally, the differentiation is applied to the function after the $\nabla$ symbol. So a more conventional notation is
\begin{align*}
\boldsymbol  A\cdot \nabla  : = (\nabla \cdot \boldsymbol  A^{\intercal})^{\intercal}, \quad \boldsymbol  A\times \nabla : = - (\nabla \times \boldsymbol  A^{\intercal})^{\intercal}.
\end{align*}
By moving the differential operator to the right, the notation is simplified and the transpose rule for matrix-vector products can be formally used. Again the right most column vector is treated as a row vector to make the notation more clean. We introduce the double differential operators as
\begin{equation*}
\inc \boldsymbol  A : = \nabla \times \boldsymbol  A \times \nabla, \quad \div\div\boldsymbol  A := \nabla\cdot \boldsymbol  A\cdot \nabla.
\end{equation*}
As the column and row operations are independent, and no chain rule is needed, the ordering of operations is not important and parentheses is skipped. Parentheses will be added when it is necessary. 

In the literature, differential operators are usually applied row-wisely to tensors. To distinguish with $\nabla$ notation, we define operators in letters as
\begin{align*}
\grad \boldsymbol  u &:= \boldsymbol  u \nabla^{\intercal} = (\partial_j u_i ) = (\nabla \boldsymbol  u)^{\intercal},\\
\curl \boldsymbol  A &: = - \boldsymbol  A\times \nabla = (\nabla \times \boldsymbol  A^{\intercal})^{\intercal},\\
\div \boldsymbol  A &: = \boldsymbol  A\cdot \nabla = (\nabla \cdot \boldsymbol  A^{\intercal})^{\intercal}.
\end{align*}
Note that the transpose operator $^{\intercal}$ is involved for tensors and in particular $\grad \boldsymbol  u \neq \nabla \boldsymbol  u$, $\curl \boldsymbol  A\neq \nabla \times \boldsymbol  A$, $\curl \boldsymbol  A\neq \boldsymbol  A\times \nabla$ and $\div \boldsymbol  A\neq \nabla \cdot \boldsymbol  A$. 
For symmetric tensors, $\div \boldsymbol  A = \boldsymbol  A\cdot \nabla,$ $\curl \boldsymbol  A = - \boldsymbol  A \times \nabla$. 


Integration by parts can be applied to row-wise differentiations as well as column-wise ones. For example, we shall frequently use
\begin{align*}
(\nabla \times \boldsymbol  \tau, \boldsymbol  \sigma)_K &= (\boldsymbol  \tau, \nabla \times \boldsymbol  \sigma)_{K} + (\boldsymbol  n\times \boldsymbol  \tau, \boldsymbol  \sigma)_{\partial K},\\
(\boldsymbol  \tau \times \nabla, \boldsymbol  \sigma)_K &= (\boldsymbol  \tau, \boldsymbol  \sigma \times \nabla)_{K} + (\boldsymbol  \tau \times \boldsymbol  n, \boldsymbol  \sigma)_{\partial K}.
\end{align*}
Similar formulae hold for $\grad, \curl, \div$ operators. Be careful on the possible sign and transpose when letter differential operators and $\nabla$ operators are mixed together. Chain rules are also better used in the same category of differential operations (row-wise, column-wise or letter operators). 

Denote the space of all  $3\times3$ matrix by $\mathbb{M}$, all symmetric $3\times3$ matrix by $\mathbb{S}$ and all skew-symmetric $3\times3$ matrix by $\mathbb{K}$. 
For any matrix $\boldsymbol  B\in \mathbb M$, we can decompose it into symmetric and skew-symmetric part as
$$
\boldsymbol  B = {\rm sym}(\boldsymbol  B) + {\rm skw}(\boldsymbol  B):= \frac{1}{2}(\boldsymbol  B + \boldsymbol  B^{\intercal}) + \frac{1}{2}(\boldsymbol  B - \boldsymbol  B^{\intercal}).
$$
The symmetric gradient of a vector function $\bs u$ is defined as
$$
\defm \boldsymbol  u := \sym \nabla \boldsymbol  u = \frac{1}{2}(\nabla \boldsymbol  u + (\nabla \boldsymbol  u)^{\intercal}) = \frac{1}{2}(\boldsymbol  u\nabla + \nabla \boldsymbol  u) .
$$
In the last identity, the dyadic product is used to emphasize the symmetry in notation. In the context of elasticity, it is commonly denoted by $\varepsilon (\bs u)$.

We define an isomorphism of $\mathbb R^3$ and the space of skew-symmetric matrices $\mathbb K$ as follows: for a vector $\boldsymbol  \omega =
( \omega_1, \omega_2, \omega_3)^{\intercal}
\in \mathbb R^3,$
\begin{equation*}
\mskw \boldsymbol  \omega :=  
\begin{pmatrix}
 0 & -\omega_3 & \omega_2 \\
\omega_3 & 0 & - \omega_1\\
-\omega_2 & \omega_1 & 0
\end{pmatrix}. 
\end{equation*}
Obviously $\mskw: \mathbb R^3 \to \mathbb K$ is a bijection. We define $\vskw: \mathbb M\to \mathbb R^3$ by $\vskw := \mskw^{-1}\circ \skw$. Using these notation, we have the decomposition
\begin{equation}
\label{eq:du} \grad \boldsymbol  v =\defm\boldsymbol  v+\frac{1}{2}\mskw(\nabla\times\boldsymbol  v).
\end{equation}

\subsection{Identities on tensors}
We shall present identities based on diagram \eqref{BGGdiagram} and refer to~\cite{Arnold;Hu:2020Complexes} for an unified proof. Let $S \boldsymbol  \tau = \boldsymbol  \tau^{\intercal} - \tr(\boldsymbol  \tau) \boldsymbol  I$ and $\iota: \mathbb R\to \mathbb M$ by $\iota v = v \boldsymbol  I$. 
\begin{equation}\label{BGGdiagram}
\begin{tikzcd}
\mathcal C^{\infty}(\mathbb{R})  \arrow{r}{\nabla} &\mathcal C^{\infty}(\mathbb R^3) \arrow{r}{\nabla \times} &\mathcal C^{\infty}(\mathbb R^3) \arrow{r}{\nabla\cdot} & \mathcal C^{\infty}(\mathbb{R})\\
\mathcal C^{\infty}(\mathbb R^3)\arrow[u,swap,"\cdot \boldsymbol  x"] \arrow{r}{\nabla} \arrow[ur, "\mathrm{id}"]&\mathcal C^{\infty}(\mathbb {M}) \arrow[u,swap,"\cdot \boldsymbol  x"] \arrow[u,swap,"\cdot \boldsymbol  x"] \arrow{r}{\nabla\times } \arrow[ur, "2\vskw"]&\mathcal C^{\infty}(\mathbb{M}) \arrow[u,swap,"\cdot \boldsymbol  x"] \arrow{r}{\nabla\cdot}\arrow[ur, "\tr"] & \mathcal C^{\infty}(\mathbb R^3) \arrow[u,swap,"\cdot \boldsymbol  x"] \\
\mathcal C^{\infty}(\mathbb R^3)\arrow[u,swap,"\times \boldsymbol  x"]\arrow{r}{\nabla} \arrow[ur,"-\mskw"]&\mathcal C^{\infty}(\mathbb{M}) \arrow[u,swap,"\times \boldsymbol  x"] \arrow{r}{\nabla\times} \arrow[ur, "S"]&\mathcal C^{\infty}(\mathbb{M}) \arrow[u,swap,"\times \boldsymbol  x"] \arrow{r}{\nabla\cdot}\arrow[ur, "2\vskw"] & \mathcal C^{\infty}(\mathbb R^3) \arrow[u,swap,"\times \boldsymbol  x"]\\
\mathcal C^{\infty}(\mathbb{R})\arrow[u,swap,"\boldsymbol  x"]\arrow{r}{\nabla} \arrow[ur, "\iota"]&\mathcal C^{\infty}(\mathbb R^3) \arrow[u,swap,"\boldsymbol  x"] \arrow{r}{\nabla\times } \arrow[ur, "-\mskw"]&\mathcal C^{\infty}(\mathbb R^3)\arrow[u,swap,"\boldsymbol  x"] \arrow{r}{\nabla\cdot}\arrow[ur, "\mathrm{id}"] & \mathcal C^{\infty}(\mathbb{R}) \arrow[u,swap,"\boldsymbol  x"].
\end{tikzcd}
\end{equation}

The north-east diagonal operator is the Poisson bracket $[\dd, \kappa] = \dd((\cdot)\kappa) - (\dd(\cdot))\kappa$ for $\dd = \nabla, \nabla \times, \nabla \cdot$ being applied from the left and the Koszul operator $\kappa = \boldsymbol  x, \times \boldsymbol  x, \cdot \boldsymbol  x$ applied from the right. For example, we have
\begin{align}
\label{eq:curlxtau} \nabla \times (\boldsymbol\tau\cdot \boldsymbol  x) - (\nabla\times \boldsymbol  \tau )\cdot \boldsymbol  x &=  2\vskw\boldsymbol\tau,\quad\;\; \text{block } (1,2),\\
\notag
\nabla (\boldsymbol  u \times \boldsymbol  x) - (\nabla \boldsymbol  u)\times \boldsymbol  x &= - \mskw \boldsymbol  u, \quad \text{block } (2,1). 
\end{align}
The parallelogram formed by the north-east diagonal and the horizontal operators is anticomutative. For example, we will use the following identities:
\begin{align}
\label{eq:trcurl}
\tr (\nabla \times \boldsymbol  \tau) &= - \nabla \cdot 2\vskw(\boldsymbol  \tau), \quad\;\; \text{block } (1,2),\\
\notag
2{\rm vskw}\nabla \boldsymbol  u &= - \nabla \times \boldsymbol  u,\\
\notag
\nabla\times \boldsymbol  u &= \nabla\cdot {\rm mskw}(\boldsymbol  u).
\end{align}
By taking transpose, we can get similar formulae for row-wise differential operators. By replacing $\partial_i$ by $x_i$, we can get the anticomutativity of the parallelgorms formed by the vertical and the north-east diagonal operators. For example, \eqref{eq:trcurl} becomes
\begin{equation}\label{eq:trxtimes}
\tr (\boldsymbol  \tau \times \bs x) = - 2\vskw(\boldsymbol  \tau)\cdot \bs x.
\end{equation}

\subsection{Tensors on surfaces}
Given a plane $F$ with normal vector $\boldsymbol  n$, for a vector $\boldsymbol  v\in \mathbb R^3$, we define its projection to plane $F$ 
$$
\Pi_F \boldsymbol  v := (\boldsymbol  n\times \boldsymbol  v)\times \boldsymbol  n = \boldsymbol  n\times (\boldsymbol  v\times \boldsymbol  n) = - \boldsymbol  n\times (\boldsymbol  n\times \boldsymbol  v) = (\bs I - \bs n\bs n^{\intercal})\bs v,
$$
which is called the tangential component of $\boldsymbol  v$. The vector $$\Pi_F^{\bot}\boldsymbol  v :=\boldsymbol  n\times \boldsymbol  v = (\bs n\times \Pi_F)\bs v$$ is called the tangential trace of $\boldsymbol  v$, which is a rotation of $\Pi_F \boldsymbol  v$ on $F$ ($90^{\circ}$ counter-clockwise with respect to $\boldsymbol  n$). Note that $\Pi_F$ is a $3\times 3$ symmetric matrix.
%
With a slight abuse of notation, we use $\Pi_F$ to denote the piece-wisely defined projection to the boundary of $K$.

We treat Hamilton operator $\nabla = (\partial_1, \partial_2, \partial_3)^{\intercal}$ as a column vector and define
$$
\nabla_F: = \Pi_F \nabla, \quad \nabla_F^{\bot} := \Pi_F^{\bot} \nabla.
$$
We have the decomposition
\begin{equation*}
\nabla = \nabla_F + \bs n \, \partial_n. 
\end{equation*}

For a scalar function $v$, 
\begin{align*}
\nabla_F v &= \Pi_F (\nabla v) = - \boldsymbol  n \times (\boldsymbol  n\times \nabla v), \\
\nabla_F^{\bot} v &= \boldsymbol  n \times \nabla v = \boldsymbol  n\times \nabla_F v,
\end{align*}
are the surface gradient of $v$ and surface $\curl$, respectively. For a vector function $\boldsymbol  v$, $\nabla_F\cdot \boldsymbol  v$ is the surface divergence:
$$
\div_F\boldsymbol  v: = \nabla_F\cdot \boldsymbol  v = \nabla_F\cdot (\Pi_F\boldsymbol  v).
$$
By the cyclic invariance of the mix product and the fact $\boldsymbol  n$ is constant, the surface rot operator is
\begin{equation}\label{eq:rotF}
{\rm rot}_F \boldsymbol  v := \nabla_F^{\bot}\cdot \boldsymbol  v = (\boldsymbol  n\times \nabla)\cdot \boldsymbol  v = \boldsymbol  n\cdot (\nabla \times \boldsymbol  v),
\end{equation}
which is the normal component of $\nabla \times \boldsymbol  v$. 
The tangential trace of $\nabla \times \boldsymbol  u$ is 
\begin{equation}\label{eq:tangentialtrace}
\boldsymbol  n\times (\nabla \times \boldsymbol  v) = \nabla (\boldsymbol  n\cdot \boldsymbol  v) - \partial_n \boldsymbol  v. 
\end{equation}
By definition,
\begin{equation*}
\nabla_F^{\bot}\cdot \boldsymbol  v = - \nabla_F\cdot (\boldsymbol  n\times \boldsymbol  v), \quad\nabla_F\cdot \boldsymbol  v = \nabla_F^{\bot}\cdot (\boldsymbol  n\times \boldsymbol  v).
\end{equation*}


When involving tensors, we define, for a vector function $\boldsymbol  v$,
\begin{align*}
&\nabla_F\boldsymbol  v:=\nabla_F\boldsymbol  v^{\intercal} = \Pi_F \nabla \boldsymbol  v^{\intercal}, \quad \grad_F \boldsymbol  v = \boldsymbol  v \nabla_F = (\nabla_F\boldsymbol  v)^{\intercal},\\
&\nabla_F^{\bot}\boldsymbol  v:=\nabla_F^{\bot}\boldsymbol  v^{\intercal} = \boldsymbol  n \times (\nabla \boldsymbol  v^{\intercal}), \quad \curl_F\boldsymbol  v:=\boldsymbol  v \nabla_F^{\bot} = ( \nabla_F^{\bot}\boldsymbol  v)^{\intercal},\\
&\defm_{F}\boldsymbol  v:= \sym \nabla_F\boldsymbol  v, \quad \sym\curl_F\boldsymbol  v :=\sym(\curl_F\boldsymbol  v).
\end{align*}
For a tensor function $\bs\tau$,
\begin{align*}
\div_F\bs\tau:=\bs\tau\cdot\nabla_F = (\nabla_F\cdot \boldsymbol  \tau^{\intercal})^{\intercal},\quad \div_F\div_F\boldsymbol  \tau: = \nabla_F\cdot \bs\tau\cdot\nabla_F,\\
\text{rot}_F\bs\tau:=\bs\tau\cdot(\boldsymbol  n\times\nabla) = (\nabla_F^{\bot}\cdot \boldsymbol  \tau^{\intercal})^{\intercal} , \quad {\rm rot}_F{\rm rot}_F\boldsymbol  \tau: = \nabla_F^{\bot}\cdot \bs\tau\cdot\nabla_F^{\bot}.
\end{align*}
Although we define the surface differentiation through the projection of differentiation of a function defined in space, it can be verified that the definition is intrinsic in the sense that it depends only on the function value on the surface $F$. Namely $\nabla_F v = \nabla_F (v|_{F}), \nabla_F\cdot \bs v = \nabla_F \cdot \Pi_F \bs v, \nabla_F \bs v = \nabla_F(\bs v|_F)$ and thus $\Pi_F$ is sometimes skipped after $\nabla_F$.

\section{Polynomial Complexes}
In this section we consider polynomial elasticity complexes on a bounded and topologically trivial domain $D\subset \mathbb R^3$ in this section.
Without loss of generality, we assume $(0,0,0) \in D$ which can be easily satisfied by changing of variable $\boldsymbol  x - \boldsymbol  x_c$ with an arbitrary $\boldsymbol  x_c\in D$ for polynomials in $D$.  

Given 
a non-negative integer $k$, 
let $\mathbb P_k(D)$ stand for the set of all polynomials in $D$ with the total degree no more than $k$, and $\mathbb P_k(D; \mathbb{X})$ denote the tensor or vector version for $\mathbb X = \mathbb S,\mathbb K, \mathbb M$, or $\mathbb R^3$. Similar notation will be applied to a two dimensional face $F$ and one dimensional edge $e$.

Recall that $\dim \mathbb P_{k}(D) = { k + d \choose d}$ for a $d$-dimensional domain $D$, $\dim \mathbb M = 9, \dim \mathbb S = 6,$ and $\dim \mathbb K = 3$. 
We list a useful result in~\cite{Chen;Huang:2020Finite}
\begin{equation}
\label{eq:radialderivativeprop1}
\mathbb P_k(D)\cap\ker(\ell+\boldsymbol x\cdot\nabla) =0
\end{equation}
for any positive number $\ell$.

\subsection{Polynomial elasticity complex}
The polynomial de Rham complex is
\begin{equation}\label{eq:polyderham}
\resizebox{.9\hsize}{!}{$
\mathbb R\autorightarrow{$\subset$}{} \mathbb P_{k+1}(D)\autorightarrow{$\nabla$}{} \mathbb P_{k}(D;\mathbb R^3)\autorightarrow{$\nabla \times$}{}\mathbb P_{k-2}(D;\mathbb R^3) \autorightarrow{$\nabla\cdot$}{} \mathbb P_{k-3}(D)\autorightarrow{}{}0
$}.
\end{equation}
As $D$ is topologically trivial, complex \eqref{eq:polyderham} is also exact, which means the range of each map is the kernel of the succeeding map. 

\begin{lemma}\label{lem:divsymtensoronto}
 $\div: \sym (\boldsymbol  x \mathbb P_{k-3}(D;\mathbb R^3)) \to \mathbb P_{k-3}(D;\mathbb R^3)$ is bijective.
\end{lemma}
\begin{proof}
As $\div(\sym (\boldsymbol  x \mathbb P_{k-3}(D;\mathbb R^3))) \subseteq \mathbb P_{k-3}(D;\mathbb R^3)$ and $\dim\sym (\boldsymbol  x \mathbb P_{k-3}(D;\mathbb R^3))=\dim\mathbb P_{k-3}(D;\mathbb R^3)$, it is sufficient to prove $\sym (\boldsymbol  x \mathbb P_{k-3}(D;\mathbb R^3))\cap\ker(\div)=\{\bs0\}$. That is: for any $\boldsymbol q\in\mathbb P_{k-3}(D;\mathbb R^3)$ satisfying $\div\sym(\boldsymbol  x \boldsymbol q^{\intercal})=\boldsymbol0$, we are going to prove $\boldsymbol q =\boldsymbol 0$. 

By direct computation,
\begin{align*}
\div (\boldsymbol  q\boldsymbol  x^{\intercal}) &= (\div \boldsymbol  x) \boldsymbol  q +  (\grad \boldsymbol  q) \cdot \boldsymbol  x = 3\boldsymbol  q + (\grad \boldsymbol  q) \cdot \boldsymbol  x,\\
\div (\boldsymbol  x\boldsymbol  q^{\intercal}) & = (\div \boldsymbol  q) \boldsymbol  x + (\grad \boldsymbol  x)\cdot \boldsymbol  q = \boldsymbol  q + (\div \boldsymbol  q) \boldsymbol  x, \\
2\div \sym(\boldsymbol x\boldsymbol q^{\intercal}) & = 4\boldsymbol  q + (\grad \boldsymbol  q) \cdot \boldsymbol  x + (\div \boldsymbol  q) \boldsymbol  x.
\end{align*}
It follows from $\div \sym(\boldsymbol  x\boldsymbol q)=\boldsymbol 0$ that 
\begin{equation}\label{eq:divsymxq}
4\boldsymbol  q + (\grad \boldsymbol  q) \cdot \boldsymbol  x = - (\div \boldsymbol  q) \boldsymbol  x.
\end{equation}
Since $\div((\grad \boldsymbol  q) \cdot \boldsymbol  x)=(1+\boldsymbol  x\cdot\grad)\div\boldsymbol  q$,
applying the divergence operator $\div$ on both side of \eqref{eq:divsymxq} yields
$$
(5+\boldsymbol  x\cdot\grad)\div\boldsymbol  q=-(3+\boldsymbol  x\cdot\grad)\div\boldsymbol  q.
$$
Hence we acquire from \eqref{eq:radialderivativeprop1} that $\div\boldsymbol  q=0$, and \eqref{eq:divsymxq} reduces to
\[
4\boldsymbol  q + (\grad \boldsymbol  q) \cdot \boldsymbol  x =\bs0 \quad \forall~\boldsymbol  x\in D.
\]
Applying \eqref{eq:radialderivativeprop1} again gives $\boldsymbol q=\boldsymbol 0$. 
\end{proof}

Recall that the linearized rigid body motion is 
\begin{equation}\label{eq:RM}
\bs{RM} = \{\boldsymbol  a \times \boldsymbol  x + \boldsymbol  b: \boldsymbol  a, \boldsymbol  b\in \mathbb R^3\} = \{\boldsymbol  N \boldsymbol  x + \boldsymbol  b: \boldsymbol  N\in \mathbb K, \boldsymbol  b\in \mathbb R^3\}.
\end{equation}

\begin{lemma}\label{lm:polycomplex}
The polynomial sequence
\begin{equation}\label{eq:elascomplex3dPoly}
\resizebox{.9\hsize}{!}{$
\boldsymbol{RM}\autorightarrow{$\subset$}{} \mathbb P_{k+1}(D;\mathbb R^3)\autorightarrow{$\defm$}{} \mathbb P_{k}(D;\mathbb S)\autorightarrow{$\inc$}{}\mathbb P_{k-2}(D;\mathbb S) \autorightarrow{$\div$}{} \mathbb P_{k-3}(D;\mathbb R^3)\autorightarrow{}{}0
$}
\end{equation}
is an exact complex.
\end{lemma}
\begin{proof}
Verification of \eqref{eq:elascomplex3dPoly} being a complex is straightforward using our notation system:
$$
\nabla \times (\nabla \boldsymbol  u + \boldsymbol  u\nabla)\times \nabla = 0, \quad (\nabla \times \boldsymbol  \tau\times \nabla)\cdot\nabla = 0.
$$

We then verify the exactness. 

\medskip
\noindent {\em 1. If $\defm (\boldsymbol  v) = 0$, then $\boldsymbol  v\in \bs{RM}$.} This is well-known and can be found in e.g.~\cite{Ciarlet:2010inequality}. 

\medskip

\noindent{\em 2. $\defm\mathbb P_{k+1}(D;\mathbb R^3)=\mathbb P_{k}(D;\mathbb S)\cap\ker(\inc)$, i.e. if $\inc(\boldsymbol  \tau) = 0$ and $\boldsymbol  \tau\in \mathbb P_{k}(D;\mathbb S)$, then there exists a $\boldsymbol  v\in \mathbb P_{k+1}(D;\mathbb R^3)$, s.t. $\boldsymbol  \tau = \defm \boldsymbol  v$}. 

As $\nabla \times (\boldsymbol  \tau \times \nabla) = 0$, we apply the exactness of de Rham complex \eqref{eq:polyderham} to each column of $\boldsymbol  \tau \times \nabla$ to conclude there exists $\boldsymbol q\in\mathbb P_{k}(D;\mathbb R^3)$ such that $
\boldsymbol  \tau \times \nabla=\nabla\boldsymbol q.$ As $\boldsymbol  \tau$ is symmetric, taking transpose, we get
$$
 - \nabla \times \boldsymbol  \tau = \boldsymbol  q\nabla. 
$$
And use \eqref{eq:trcurl} to conclude
$$
\nabla\cdot \boldsymbol  q = \tr (\nabla \boldsymbol  q) =\tr ( \boldsymbol  \tau \times \nabla) = - \tr (\nabla \times \boldsymbol  \tau) = \nabla \cdot 2\vskw(\boldsymbol  \tau) = 0.
$$
Hence there exists $\boldsymbol q_1\in\mathbb P_{k+1}(D;\mathbb R^3)$ such that
\[
\boldsymbol q=\nabla \times \boldsymbol q_1,
\]
which implies
\[
-\nabla \times \boldsymbol  \tau = \boldsymbol  q \nabla =  (\nabla \times \boldsymbol  q_1) \nabla = \nabla \times ( \boldsymbol  q_1 \nabla). 
\]
i.e. $\nabla \times (\boldsymbol \tau + \boldsymbol q_1\nabla)=0$. Then there exists $\boldsymbol q_2\in\mathbb P_{k+1}(D;\mathbb R^3)$ such that
\[
\boldsymbol \tau + \boldsymbol q_1 \nabla = \nabla\boldsymbol q_2.
\]
Then, as $\boldsymbol  \tau$ is symmetric,
$$\boldsymbol \tau=\sym\boldsymbol \tau=\sym(\nabla\boldsymbol q_2 - \boldsymbol q_1 \nabla )=\defm(\boldsymbol q_2 - \boldsymbol q_1)\in\defm\mathbb P_{k+1}(D;\mathbb R^3).$$

\medskip

\noindent{\em 3. $\div\mathbb P_{k-2}(D;\mathbb S)=\mathbb P_{k-3}(D;\mathbb R^3)$ holds from Lemma~\ref{lem:divsymtensoronto}.} 

\medskip

\noindent{\em 4. $\inc\mathbb P_{k}(D;\mathbb S)=\mathbb P_{k-2}(D;\mathbb S)\cap\ker(\div)$.}

Obviously $\inc\mathbb P_{k}(D;\mathbb S)\subseteq\mathbb P_{k-2}(D;\mathbb S)\cap\ker(\div)$. Then it suffices to show the dimensions of these two subspaces are equal. Recall that for a linear operator $T$ defined on a finite dimensional linear space $V$, we have the relation
\begin{equation}\label{eq:dim}
\dim V = \dim \ker(T) + \dim \img(T). 
\end{equation}

As $\div$ is surjective shown in Step 3, by \eqref{eq:dim},
$$
\dim\mathbb P_{k-2}(D;\mathbb S)\cap\ker(\div)=\dim\mathbb P_{k-2}(D;\mathbb S)-\dim\mathbb P_{k-3}(D;\mathbb R^3)=\frac{1}{2}(k+4)(k^2-k).
$$
By results in Step 1 and 2, we count the dimension
\begin{align*}
\dim (\inc \mathbb P_{k}(D; \mathbb S)) &= \dim (\mathbb P_{k}(D; \mathbb S)) - \dim (\defm \mathbb P_{k+1}(D; \mathbb R^3)) \\
&= \dim (\mathbb P_{k}(D; \mathbb S)) - \dim (\mathbb P_{k+1}(D; \mathbb R^3)) + \dim \bs{RM} \\
& = \frac{1}{2}(k+4)(k^2-k).
\end{align*}
%
%
Then the desired result follows. 
\end{proof}

\subsection{Koszul elasticity complex}
 Recall the Koszul complex
{\small \begin{equation}\label{eq:Koszul}
0
\autorightarrow{}{}
\mathbb P_{k-2}(D) 
\autorightarrow{$\phi\boldsymbol  x$}{} 
\mathbb P_{k-1}(D;\mathbb R^3)
\autorightarrow{$\boldsymbol  u\times \boldsymbol  x$}{} 
\mathbb P_{k}(D;\mathbb R^3)
\autorightarrow{$\boldsymbol  v\cdot \boldsymbol  x$}{} \mathbb P_{k+1}(D) \longrightarrow 0.
\end{equation}}

Define operator $\boldsymbol \pi_{RM}: \mathcal C^1(D; \mathbb R^3)\to \boldsymbol{RM}$ as
\[
\boldsymbol \pi_{RM}\boldsymbol  v:=\boldsymbol  v(0,0,0)+\frac{1}{2}(\curl\boldsymbol  v)(0,0,0)\times\boldsymbol  x.
\]
By direct calculation $\nabla \times (\boldsymbol  a \times \boldsymbol  x) = 2\boldsymbol  a$ and definition of $\bs{RM}$ cf. \eqref{eq:RM}, it holds
\begin{equation}\label{eq:piRTprop}
\boldsymbol \pi_{RM}\boldsymbol  v=\boldsymbol  v\quad \forall~\boldsymbol  v\in\boldsymbol{RM}.
\end{equation}

\begin{lemma}
The following polynomial and operators sequences 
\begin{equation}\label{eq:elas3dKoszulcomplexPoly}
\resizebox{0.92\hsize}{!}{$
0\stackrel{\subset}{\longrightarrow}\mathbb P_{k-3}(D;\mathbb R^3) \stackrel{\sym(\boldsymbol v \boldsymbol x^{\intercal})}{\longrightarrow}
\mathbb P_{k-2}(D;\mathbb S) \stackrel{\boldsymbol x\times \boldsymbol\tau \times \boldsymbol x}{\longrightarrow}
\mathbb P_k(D;\mathbb S) \stackrel{\boldsymbol\tau\cdot \boldsymbol x}{\longrightarrow}
\mathbb P_{k+1}(D;\mathbb R^3)\stackrel{\boldsymbol \pi_{RM}}{\longrightarrow}\boldsymbol{RM}
\stackrel{}{\longrightarrow}0
$}
\end{equation}
is a complex and is exact.
\end{lemma}
\begin{proof}
It is similar to the proof of Lemma \ref{lm:polycomplex} and symbolically replace $\nabla = (\partial_1, \partial_2, \partial_3)^{\intercal}$ by $\boldsymbol  x = (\boldsymbol  x_1, \boldsymbol  x_2, \boldsymbol  x_3)^{\intercal}$. 
We first verify \eqref{eq:elas3dKoszulcomplexPoly} is a complex. For any $\boldsymbol v\in\mathbb P_{k-3}(D;\mathbb R^3)$ and $\boldsymbol\tau\in\mathbb P_{k-2}(D;\mathbb S)$, we have
\begin{align*}
\boldsymbol x\times \sym(\boldsymbol v\boldsymbol x^{\intercal})\times \boldsymbol x &=\frac{1}{2}\boldsymbol x\times(\boldsymbol x\boldsymbol v^{\intercal} + \boldsymbol  v\boldsymbol  x^{\intercal})\times \boldsymbol  x=\boldsymbol0,\\
(\boldsymbol x\times\boldsymbol\tau \times \boldsymbol  x)\cdot \boldsymbol  x &=\boldsymbol 0.
\end{align*}
As $\boldsymbol\tau\in\mathbb P_{k}(D;\mathbb S)$, \eqref{eq:curlxtau} implies
$\nabla \times( \boldsymbol  \tau\cdot \boldsymbol  x) =(\nabla\times \boldsymbol\tau )\cdot \boldsymbol  x,$
we get
$\boldsymbol \pi_{RM}(\boldsymbol\tau\cdot \boldsymbol  x)= \boldsymbol  0.$
Thus \eqref{eq:elas3dKoszulcomplexPoly} is a complex.

We now verify the exactness.  

\medskip
\noindent{\em 1. If $\boldsymbol x\times\boldsymbol\tau \times \boldsymbol x=\boldsymbol  0$ and $\boldsymbol  \tau \in \mathbb P_{k-2}(D; \mathbb S)$, then $\boldsymbol  \tau = \sym (\boldsymbol  v\boldsymbol  x^{\intercal})$ for some $\boldsymbol  v\in \mathbb P_{k-3}(D;\mathbb R^3)$.}

For any $\boldsymbol \tau\in\mathbb P_{k-2}(D;\mathbb S)$ satisfying $\boldsymbol x\times(\boldsymbol\tau \times \boldsymbol x)=\boldsymbol  0$, by the exactness of Koszul complex \eqref{eq:Koszul}, there exists $\widetilde{\boldsymbol v}\in \mathbb P_{k-2}(D;\mathbb R^3)$ such that $\boldsymbol\tau \times \boldsymbol x=\boldsymbol x\widetilde{\boldsymbol v}$.
By \eqref{eq:trxtimes}, as $\bs \tau$ is symmetric, $\boldsymbol\tau \times \boldsymbol x$ is trace-free. Then it follows $\widetilde{\boldsymbol v}\cdot \boldsymbol  x = \tr(\boldsymbol  x\widetilde{\boldsymbol v} ) = \tr (\boldsymbol  \tau \times \boldsymbol  x) = 0$. Then there exists $\boldsymbol  v_1\in\mathbb P_{k-3}(D;\mathbb R^3)$ such that $\widetilde{\boldsymbol v}=\boldsymbol  v_1\times \boldsymbol x$. As a result, we have
$$
(\boldsymbol\tau-\boldsymbol x\boldsymbol  v_1)\times \boldsymbol x=\boldsymbol\tau \times \boldsymbol x-\boldsymbol x(\boldsymbol  v_1\times \boldsymbol x)= \boldsymbol\tau \times \boldsymbol x-\boldsymbol x \widetilde{\boldsymbol v} = 0.
$$
Again there exists $\boldsymbol v_2\in \mathbb P_{k-3}(D;\mathbb R^3)$ such that $\boldsymbol\tau=\boldsymbol x\boldsymbol v_1+\boldsymbol v_2\boldsymbol x$. By the symmetry of $\boldsymbol\tau$, it holds $\boldsymbol\tau=\sym(\boldsymbol  x(\boldsymbol v_1+\boldsymbol v_2))$. 

\medskip
\noindent{\em 2. If $\boldsymbol  \tau\cdot \boldsymbol  x = 0$ and $\boldsymbol  \tau \in \mathbb P_k(D;\mathbb S)$, then $\boldsymbol  \tau = \boldsymbol  x\times \bs\sigma \times \boldsymbol  x$ for some $\bs\sigma \in \mathbb P_{k-2}(D;\mathbb S)$.}

For any $\boldsymbol \tau\in\mathbb P_k(D;\mathbb S)$ satisfying $\boldsymbol \tau\cdot \boldsymbol x=\boldsymbol  0$, by the exactness of Koszul complex \eqref{eq:Koszul}, there exists $\boldsymbol \tau_1\in \mathbb P_{k-1}(D;\mathbb M)$ such that $\boldsymbol\tau=\boldsymbol\tau_1\times\boldsymbol x$. By the symmetry of $\boldsymbol \tau$, it holds
\[
(\boldsymbol  x\cdot \boldsymbol  \tau_1)\times\boldsymbol x= \boldsymbol  x\cdot (\boldsymbol  \tau_1\times\boldsymbol x)
=\boldsymbol  x\cdot \boldsymbol\tau=(\boldsymbol\tau\cdot \boldsymbol x)^{\intercal}= \boldsymbol0.
\]
Thus there exists $q\in\mathbb P_{k-1}(D)$ satisfying $\boldsymbol  x\cdot \boldsymbol  \tau_1=q\boldsymbol x$, i.e. $\boldsymbol  x\cdot (\boldsymbol\tau_1 -q\boldsymbol I)=\boldsymbol 0$. Again there exists $\boldsymbol\tau_2\in\mathbb P_{k-2}(D;\mathbb M)$ satisfying $\boldsymbol\tau_1=q\boldsymbol I + \boldsymbol x\times \boldsymbol\tau_2$. Hence
\[
\boldsymbol\tau=q\boldsymbol I\times\boldsymbol x + \boldsymbol x\times  \boldsymbol\tau_2 \times \boldsymbol x.
\]
It follows from the symmetry of $\boldsymbol\tau$ that
\begin{align*}
\boldsymbol\tau&=\sym\boldsymbol\tau=\sym(q\boldsymbol I\times\boldsymbol x + \boldsymbol x\times  \boldsymbol\tau_2 \times \boldsymbol x)= \sym(\boldsymbol x\times \boldsymbol\tau_2 \times \boldsymbol x) =\boldsymbol x\times\sym\boldsymbol\tau_2 \boldsymbol  \times \boldsymbol  x.
\end{align*}
Here we use the fact that $\boldsymbol x\times\skw\boldsymbol\tau_2\times \boldsymbol  x \in\mathbb P_{k}(D;\mathbb K)$.

\medskip
\noindent {\rm 3. $\mathbb P_{k}(D;\mathbb S) \cdot \boldsymbol  x = \mathbb P_{k+1}(D;\mathbb R^3)\cap \ker(\pi_{RM})$.}

As a result of step 1,
\begin{align*}
\dim(\boldsymbol x\times\mathbb P_{k-2}(D;\mathbb S)\times \boldsymbol  x)=\dim\mathbb P_{k-2}(D;\mathbb S)-\dim\mathbb P_{k-3}(D;\mathbb R^3) =\frac{1}{2}k(k-1)(k+4). 
\end{align*}
Then we get from step 2 that 
\begin{align}
\dim(\mathbb P_{k}(D;\mathbb S)\cdot \boldsymbol x)&=\dim\mathbb P_{k}(D;\mathbb S)-\dim(\boldsymbol x\times\mathbb P_{k-2}(D;\mathbb S) \times \boldsymbol  x)\notag\\
&=(k+3)(k+2)(k+1)-\frac{1}{2}k(k-1)(k+4) \notag\\
&=\frac{1}{2}(k+4)(k+3)(k+2)-6. \label{eq:20200506-2}
\end{align}
It follows from \eqref{eq:piRTprop} that $\boldsymbol \pi_{RM}\mathbb P_{k+1}(D;\mathbb R^3)=\boldsymbol{RM}$, and by \eqref{eq:20200506-2},
\[
\dim(\mathbb P_{k}(D;\mathbb S)\cdot \boldsymbol x)+\dim\boldsymbol{RM}=\dim\mathbb P_{k+1}(D;\mathbb R^3).
\]
Therefore the complex \eqref{eq:elas3dKoszulcomplexPoly} is exact.
\end{proof}

\begin{remark}
 Another Koszul elasticity complex is constructed in~\cite[Section~3.2]{ChristiansenHuSande2020} by using different Koszul operators which satisfies homotopy identities. Ours is simpler and sufficient to derive the required decomposition.    
\end{remark}
%

\subsection{Decomposition of polynomial tensor spaces}

Combining the two complexes \eqref{eq:elascomplex3dPoly} and \eqref{eq:elas3dKoszulcomplexPoly} yields
\begin{equation}\label{eq:elascomplex3dPolydouble}
\resizebox{.92\hsize}{!}{$
\xymatrix{
\boldsymbol{RM}\ar@<0.4ex>[r]^-{\subset} & \mathbb P_{k+1}(D;\mathbb R^3)\ar@<0.4ex>[r]^-{\defm}\ar@<0.4ex>[l]^-{\pi_{RM}} & \mathbb P_{k}(D;\mathbb S)\ar@<0.4ex>[r]^-{\inc}\ar@<0.4ex>[l]^-{\boldsymbol\tau\cdot \boldsymbol x}  & \mathbb P_{k-2}(D;\mathbb S) \ar@<0.4ex>[r]^-{{\div}}\ar@<0.4ex>[l]^-{\boldsymbol x\times\boldsymbol\tau \times \boldsymbol x} & \mathbb P_{k-3}(D;\mathbb R^3)  \ar@<0.4ex>[r]^-{} \ar@<0.4ex>[l]^-{\sym(\boldsymbol v \boldsymbol x^{\intercal})}
& 0 \ar@<0.4ex>[l]^-{\supset} }.
$}
\end{equation}
Although no homotopy identity, from \eqref{eq:elascomplex3dPolydouble}, we can derive the following space decompositions which play an vital role in the design of degree of freedoms. 
\begin{lemma}
We have the following space decompositions
\begin{align}
\label{eq:polydecomp0}
\mathbb P_{k+1}(D;\mathbb R^3)&=(\mathbb P_k(D;\mathbb S)\cdot \boldsymbol x)\oplus\boldsymbol{RM},\\
\label{eq:polydecomp1}
\mathbb P_k(D;\mathbb S ) &=\defm\mathbb P_{k+1}(D;\mathbb R^3)\oplus(\boldsymbol x\times\mathbb P_{k-2}(D;\mathbb S)\times \boldsymbol  x),\\
\label{eq:polydecomp2}
\mathbb P_{k-2}(D;\mathbb S)&=\inc \mathbb P_k(D;\mathbb S)\oplus \sym(\mathbb P_{k-3}(D;\mathbb R^3)\boldsymbol  x).
\end{align}
\end{lemma}
\begin{proof}
The decomposition \eqref{eq:polydecomp0} is trivial by the exactness of \eqref{eq:elas3dKoszulcomplexPoly}. 

For any $\boldsymbol q\in\mathbb P_{k+1}(D;\mathbb R^3)$ satisfying $\defm\boldsymbol q\in\boldsymbol x\times\mathbb P_{k-2}(D;\mathbb S)\times \boldsymbol  x$, we have
\[
(\nabla\boldsymbol q+(\nabla\boldsymbol q)^{\intercal})\cdot \boldsymbol x=2(\defm\boldsymbol q)\cdot\boldsymbol x=\boldsymbol 0.
\]
Since $(\nabla\boldsymbol q)\boldsymbol x=\nabla(\boldsymbol x^{\intercal}\boldsymbol q)-\boldsymbol q$, we get
\begin{equation}\label{eq:20200509}
(\nabla\boldsymbol q)^{\intercal}\cdot\boldsymbol  x + \nabla(\boldsymbol x^{\intercal}\boldsymbol q)=\boldsymbol q.
\end{equation}
Noting that
\[
\boldsymbol x\cdot(\nabla\boldsymbol q)^{\intercal}\cdot\boldsymbol  x=\boldsymbol x\cdot(\defm\boldsymbol q)\cdot\boldsymbol  x=0,
\]
we obtain from \eqref{eq:20200509} that
\[
(\boldsymbol x\cdot\nabla)(\boldsymbol x^{\intercal}\boldsymbol q)=\boldsymbol x^{\intercal}\boldsymbol q.
\]
Hence $\boldsymbol x^{\intercal}\boldsymbol q$ is a linear function. In turn, it follows from \eqref{eq:20200509} that $\boldsymbol q\in\mathbb P_{1}(D;\mathbb R^3)$, which together with the fact $\boldsymbol x^{\intercal}\boldsymbol q$ is linear implies $\boldsymbol q\in\boldsymbol{RM}$. Thus \eqref{eq:polydecomp1} follows from the fact that the dimensions on two sides of \eqref{eq:polydecomp1} coincide.

By Lemma~\ref{lem:divsymtensoronto}, the sum in \eqref{eq:polydecomp2} is a direct sum. Thus the decomposition \eqref{eq:polydecomp2} follows.
\end{proof}

\subsection{Polynomial complexes in two dimensions}
We have similar polynomial complexes in two dimensions. Here we collect some which will appear as the trace complex on face $F$ of a polyhedron. Let $\boldsymbol  n$ be a normal vector of $F$. For $\boldsymbol  x\in F$, denote by $\boldsymbol  x^{\bot} = \boldsymbol  n\times \boldsymbol  x$. 
Set $\boldsymbol{RT}:=\mathbb P_0(F;\mathbb R^2)+\boldsymbol  x\mathbb P_0(F)$.
For a scalar function $v$,
$$
\pi_1v:=v(0,0)+\boldsymbol  x\cdot\nabla_Fv(0,0).
$$
Again, here without loss of generality, we assume $(0,0)\in F$ and in general the $\boldsymbol  x$ in the results presented below can be replaced by $\boldsymbol  x - \boldsymbol  x_c$ with an arbitrary $\boldsymbol  x_c\in F$.  

The following $\div\div$ polynomial complexes has been established in~\cite{ChenHuang2020}:
\begin{equation}\label{eq:divdivcomplexPolydouble}
\xymatrix{
\boldsymbol{RT}\ar@<0.4ex>[r]^-{\subset} & \; \mathbb P_{k+1}(F;\mathbb R^2)\; \ar@<0.4ex>[r]^-{\sym\curl_F}\ar@<0.4ex>[l]^-{\boldsymbol x}  & \; \mathbb P_k(F;\mathbb S) \ar@<0.4ex>[r]^-{\div_F{\div}_F}\; \ar@<0.4ex>[l]^-{\boldsymbol  \tau\cdot \boldsymbol x^{\bot}} & \; \mathbb P_{k-2}(F)  \; \ar@<0.4ex>[r]^-{} \ar@<0.4ex>[l]^-{ \boldsymbol x \boldsymbol x^{\intercal}v}
& 0 \ar@<0.4ex>[l]^-{\supset} },
\end{equation}
which implies the following decomposition
\begin{itemize}
 \item $\mathbb P_{k+2}(F; \mathbb R^2)= (\mathbb P_{k+1}(F; \mathbb S)\cdot\boldsymbol x^{\bot})\oplus\boldsymbol{RT}.$
 \smallskip 
 \item $\displaystyle
\mathbb P_{k}(F; \mathbb S)=\sym\curl_F \, \mathbb P_{k+1}(F; \mathbb R^2) \oplus \mathbb P_{k-2}(F)\boldsymbol  x\boldsymbol  x^{\intercal}.
$

 \smallskip
 \item $\div_F\div_F: \mathbb P_{k-2}(F)\boldsymbol  x\boldsymbol  x^{\intercal}\to\mathbb P_{k-2}(F)$ is a bijection.
\end{itemize}
The following two dimensional Hessian polynomial complex and its Koszul complex can be also found in~\cite[Section 3.1]{ChenHuang2020}
\begin{equation}\label{eq:hessiancomplexPolydouble}
\xymatrix{
\mathbb P_1(F)\ar@<0.4ex>[r]^-{\subset} & \; \mathbb P_{k+2}(F)\; \ar@<0.4ex>[r]^-{\nabla_F^2}\ar@<0.4ex>[l]^-{\pi_1}  & \; \mathbb P_k(F;\mathbb S) \ar@<0.4ex>[r]^-{\rot_F}\; \ar@<0.4ex>[l]^-{\boldsymbol  x\cdot \boldsymbol  \tau \cdot \boldsymbol x} & \; \mathbb P_{k-1}(F; \mathbb R^2)  \; \ar@<0.4ex>[r]^-{} \ar@<0.4ex>[l]^-{\sym(\boldsymbol  x^{\perp}\boldsymbol  v )}
& 0 \ar@<0.4ex>[l]^-{\supset} }.
\end{equation}
The implied decompositions are
\begin{itemize}
 \item $\mathbb P_{k+2}(F)= (\boldsymbol x\cdot\mathbb P_{k}(F; \mathbb S) \cdot \boldsymbol  x) \oplus\mathbb P_1(F).$
 \smallskip 
 \item $\displaystyle
\mathbb P_{k}(F; \mathbb S)=\nabla_F^2 \, \mathbb P_{k+2}(F) \oplus \sym (\boldsymbol x^{\perp}\mathbb P_{k-1}(F; \mathbb R^2)).
$

 \smallskip
 \item $ \rot_F: \sym (\boldsymbol x^{\perp}\mathbb P_{k-1}(F; \mathbb R^2))\to\mathbb P_{k-1}(F; \mathbb R^2)$ is a bijection.
\end{itemize}

\section{Traces and Bubble Complexes}
Besides the decomposition of polynomial spaces, another key of our construction  is the characterization of the trace operator. We first derive a symmetric form of Green's identity for the $\inc$ operator from which we define two traces. We show the traces of spaces in the elasticity complex form two complexes on each face and will call them trace complexes. On the other hand, the kernel of traces in the polynomial space are called bubble function spaces which also form a complex and is called the bubble complexes. We also present several bubble complexes on each face. 

When define and study the traces, we consider smooth enough functions not in the most general Sobolev spaces setting. The precise Sobolev spaces for the traces of the $\inc$ operator are not easy to identity and not necessary as the shape function is polynomial which is smooth inside one element. 
 
\subsection{Green's identity}
Consider $\boldsymbol  \sigma, \boldsymbol   \tau \in\boldsymbol  H^2(K;\mathbb S)$. By the symbolical symmetry, we expect the following symmetric form of Green's identity
\begin{equation*}
(\nabla \times\boldsymbol  \sigma \times \nabla, \boldsymbol  \tau)_K - (\nabla \times \boldsymbol  \tau \times \nabla, \boldsymbol  \sigma)_K = ({\rm tr}_1(\boldsymbol  \sigma), {\rm tr}_2(\boldsymbol  \tau))_{\partial K} - ({\rm tr}_1(\boldsymbol  \tau), {\rm tr}_2(\boldsymbol  \sigma))_{\partial K},
\end{equation*}
which belongs to a class of second Green's identities. For the scalar Laplacian operator, it reads as: for $u, v\in H^2(K)$,
$$
- (\Delta u, v)_K + (\Delta v, u)_K = ({\rm tr}_1(u), {\rm tr}_2(v))_{\partial K} - ({\rm tr}_1(v), {\rm tr}_2(u))_{\partial K}.
$$
where ${\rm tr}_1(u) = u$ is the Dirichlet trace and ${\rm tr}_2(v) = \partial_n v$ is the Neumann trace. 
For the $\curl$ operator, we have a similar formula: for $\boldsymbol  u, \boldsymbol  v\in \boldsymbol  H^2(K)$,
$$
(\nabla \times \nabla \times \boldsymbol  u, \boldsymbol  v)_K - (\nabla \times \nabla \times \boldsymbol  v, \boldsymbol  u)_K =  - ({\rm tr}_1(\boldsymbol  u), {\rm tr}_2(\boldsymbol  v))_{\partial K}+({\rm tr}_1(\boldsymbol  v),{\rm tr}_2(\boldsymbol  u))_{\partial K}.
$$
where ${\rm tr}_1(\boldsymbol  u) = (\boldsymbol  n\times \boldsymbol  u)\times \boldsymbol  n $ is the tangential component of $\boldsymbol  u$ (Dirichlet type) and ${\rm tr}_2(\boldsymbol  u) = \boldsymbol  n\times (\nabla \times \boldsymbol  u)$ is the Neumann type trace. 


As $\boldsymbol  \sigma$ is symmetric, $(\nabla\times \boldsymbol  \sigma)^{\intercal} = - \boldsymbol  \sigma \times \nabla$. Therefore $(\nabla \times (\cdot), (\cdot)\times \nabla)$ is a  symmetric bilinear form on $\boldsymbol  H^1(K,\mathbb S)$, i.e.,
   $$
   (\nabla \times \boldsymbol   \sigma, \boldsymbol  \tau \times \nabla)_K = ( \nabla \times \boldsymbol  \tau, \boldsymbol  \sigma \times \nabla)_K.
   $$
Applying integration by parts, we have
\begin{align}
\label{eq:firstsigma}   (\nabla \times \boldsymbol  \sigma, \boldsymbol  \tau \times \nabla)_K = (\nabla \times\boldsymbol  \sigma \times \nabla, \boldsymbol  \tau)_K + (\nabla \times \boldsymbol  \sigma, \boldsymbol  \tau\times \boldsymbol   n)_{\partial K},\\
\label{eq:firsttau}   (\nabla \times \boldsymbol  \tau,\boldsymbol  \sigma \times \nabla)_K = (\nabla \times \boldsymbol  \tau \times \nabla,\boldsymbol  \sigma)_K + (\nabla \times \boldsymbol  \tau, \boldsymbol  \sigma\times \boldsymbol   n)_{\partial K}.
\end{align}
The difference between \eqref{eq:firstsigma} and \eqref{eq:firsttau} implies the first symmetric Green's identity
\begin{equation*}
   (\nabla \times \bs\sigma \times \nabla, \boldsymbol  \tau)_K - (\boldsymbol  \sigma,\nabla \times \bs\tau \times \nabla)_K =  (\boldsymbol  \sigma\times \boldsymbol  n, \nabla \times \boldsymbol  \tau)_{\partial K} - (\nabla \times \boldsymbol  \sigma, \boldsymbol  \tau\times \boldsymbol   n)_{\partial K}.
\end{equation*}
But in this form, the trace $\boldsymbol  \sigma\times \boldsymbol  n$ and $\nabla \times \boldsymbol  \sigma$ are still linearly dependent. 

We further expand the boundary term into tangential and normal parts
   $$
   (\boldsymbol  \sigma \times \boldsymbol  n, \nabla \times \boldsymbol  \tau)_{\partial K} = (\boldsymbol  n \times\boldsymbol  \sigma \times \boldsymbol  n, \boldsymbol  n \times (\nabla \times \boldsymbol  \tau))_{\partial K} + (\boldsymbol  n \cdot\boldsymbol  \sigma \times \boldsymbol  n, \boldsymbol  n\cdot (\nabla \times \boldsymbol  \tau))_{\partial K}.
   $$
Recall that, on one face $F\in \partial K$,
   $\boldsymbol  n\cdot (\nabla \times \boldsymbol  \tau) = \nabla_F^{\bot}\cdot \Pi_{F}\boldsymbol  \tau = -\nabla_F\cdot (\boldsymbol  n\times \boldsymbol  \tau)$.  Then integration by parts on face $F$, we get 
   \begin{align*}
   (\boldsymbol  n \cdot\boldsymbol  \sigma \times \boldsymbol  n, \nabla_F\cdot (\boldsymbol  n\times \boldsymbol  \tau))_{F} &= - (\nabla_F(\boldsymbol  n\cdot\boldsymbol  \sigma \times \boldsymbol  n), \boldsymbol  n\times \boldsymbol  \tau)_F \\
   &\quad + \sum_{e\in \partial F}(\boldsymbol  n \cdot\boldsymbol  \sigma \times \boldsymbol  n, \boldsymbol  n_{F,e}\cdot (\boldsymbol  n\times \boldsymbol  \tau))_{e}\\
   &= (\nabla_F(\boldsymbol  n\cdot\boldsymbol  \sigma \ \Pi_F), \boldsymbol  n\times \boldsymbol  \tau \times \boldsymbol  n)_F \\
   &\quad - \sum_{e\in \partial F}(\boldsymbol  n \cdot\boldsymbol  \sigma \times \boldsymbol  n, \boldsymbol  t_{F,e}\cdot \boldsymbol  \tau)_{e},
   \end{align*}
   where recall that $\boldsymbol  t_{F,e} = \boldsymbol  n\times \boldsymbol  n_{F,e}$.
Therefore we can write the boundary term as    
\begin{align*}
   (\boldsymbol  \sigma \times \boldsymbol  n, \nabla \times \boldsymbol  \tau)_{\partial K} &= (\boldsymbol  n \times\boldsymbol  \sigma \times \boldsymbol  n, \boldsymbol  n \times (\nabla \times \boldsymbol  \tau))_{\partial K} \\
   &- \sum_{F\in \mathcal F(K)}(\nabla_F(\boldsymbol  n\cdot\boldsymbol  \sigma \ \Pi_F), \boldsymbol  n\times \boldsymbol  \tau \times \boldsymbol  n)_F\\
   &\quad + \sum_{F\in \mathcal F(K)}\sum_{e\in \partial F}(\boldsymbol  n \cdot\boldsymbol  \sigma \times \boldsymbol  n, \boldsymbol  t_{e}\cdot \boldsymbol  \tau)_{e},
\end{align*}
   and by symmetry
\begin{align*}
   (\boldsymbol  \tau \times \boldsymbol  n, \nabla \times \boldsymbol  \sigma)_{\partial K} &= (\boldsymbol  n \times \boldsymbol  \tau \times \boldsymbol  n, \boldsymbol  n \times (\nabla \times\boldsymbol   \sigma))_{\partial K} \\
   &- \sum_{F\in \mathcal F(K)}(\nabla_F(\boldsymbol  n\cdot \boldsymbol  \tau \ \Pi_F), \boldsymbol  n\times\boldsymbol  \sigma \times\boldsymbol   n)_F\\
   &\quad + \sum_{F\in \mathcal F(K)}\sum_{e\in \partial F}(\boldsymbol  n \cdot\boldsymbol  \tau \times \boldsymbol  n, \boldsymbol  t_{e}\cdot \boldsymbol  \sigma)_{e}.
\end{align*}
The difference of these two terms suggests us to define
\begin{align*}
{\rm tr}_1(\boldsymbol  \tau) &:=  \boldsymbol  n\times\boldsymbol  \tau \times \boldsymbol  n, \\
\widetilde{{\rm tr}}_2( \boldsymbol  \tau) &:= \boldsymbol  n\times (\nabla \times \boldsymbol  \tau)\Pi_F + \nabla_F(\boldsymbol  n\cdot \boldsymbol  \tau \ \Pi_F).
\end{align*}
We can simplify the trace $\widetilde{{\rm tr}}_2(\boldsymbol  \tau)$ as follows. Apply $\Pi_F(\cdot)\Pi_F$ to the tangential trace of $\nabla \times \boldsymbol  \tau$ cf. \eqref{eq:tangentialtrace} to get
\begin{equation}\label{eq:tangentialcurl}
\Pi_F(\boldsymbol   n \times (\nabla \times \boldsymbol  \tau)) \Pi_F = \nabla_F (\boldsymbol  n\cdot \boldsymbol  \tau \Pi_F) - \Pi_F \partial_n \boldsymbol  \tau \Pi_F.
\end{equation}
   Because  $\widetilde{{\rm tr}}_2(\boldsymbol  \tau)$ is integrated on the face with a tangential symmetric matrix $\boldsymbol  n \times\boldsymbol  \sigma \times \boldsymbol  n$, it can be further simplified to $\sym \widetilde{{\rm tr}}_2( \boldsymbol  \tau)$. Therefore we define
\begin{equation}\label{eq:tr2}
   {\rm tr}_2(\boldsymbol  \tau) := \sym \widetilde{{\rm tr}}_2( \boldsymbol  \tau) = 2{\rm def}_F(\boldsymbol  n\cdot \boldsymbol  \tau \, \Pi_F) - \Pi_F \partial_n \, \boldsymbol  \tau \, \Pi_F,
\end{equation}
which is a symmetric matrix on each face. Such trace has been identified in~\cite{ArnoldAwanouWinther2008}.

We present another form of ${\rm tr}_2$ which is obtained by taking the transpose of the second term in $\widetilde{\rm tr}_2(\boldsymbol  \tau)$ and more useful than \eqref{eq:tr2}.
\begin{lemma}
For any sufficiently smooth and symmetric tensor $\bs\tau$, it holds
\begin{align}\label{eq:tr2prop1}
{\rm tr}_2(\boldsymbol  \tau) &=\boldsymbol  n\times(\nabla\times\bs\tau)\Pi_F + (\Pi_F\boldsymbol  \tau\cdot \boldsymbol  n  )\nabla_F \\
\label{eq:tr2prop2}&= \Pi_F(\boldsymbol  \tau\times\nabla)\times\boldsymbol  n + \nabla_F(\boldsymbol  n\cdot \boldsymbol  \tau \ \Pi_F).
\end{align}
\end{lemma}
\begin{proof}
We take the transpose of \eqref{eq:tangentialcurl} and use the symmetry of $\boldsymbol  \tau$ to get
\begin{equation}\label{eq:columncurl}
\Pi_F ((\boldsymbol  \tau \times \nabla)\times \boldsymbol  n) \Pi_F = (\Pi_F \boldsymbol  \tau \cdot \boldsymbol  n)\nabla_F - \Pi_F\partial_n \boldsymbol  \tau\Pi_F.
\end{equation}
The difference of \eqref{eq:columncurl} and \eqref{eq:tangentialcurl} implies
$$
\Pi_F (\boldsymbol  n\times (\nabla \times \boldsymbol  \tau)) \Pi_F + (\Pi_F\boldsymbol  \tau\cdot \boldsymbol  n)\nabla_F = \Pi_F ((\boldsymbol  \tau \times \nabla)\times \boldsymbol  n) \Pi_F + \nabla_F (\boldsymbol  n \cdot \boldsymbol  \tau \Pi_F).
$$
As ${\rm tr}_2(\boldsymbol  \tau) = \sym \widetilde{\rm tr}_2(\boldsymbol  \tau)$, we obtain \eqref{eq:tr2prop1}. As $\boldsymbol  \tau$ is symmetric, taking transpose, we obtain \eqref{eq:tr2prop2}.
\end{proof}
 
We are in the position to summarize the symmetric form of Green's identity.
\begin{theorem}[Symmetric Green's identity for the $\inc$ operator] Let $K$ be a polyhedron, and let $\boldsymbol \sigma, \boldsymbol  \tau\in \boldsymbol  H^2(K; \mathbb S)$. Then we have
 \begin{align}
   (\nabla \times \bs\sigma \times \nabla, \boldsymbol  \tau)_K - (\boldsymbol  \sigma,\nabla \times \bs\tau \times \nabla)_K &= ({\rm tr}_1(\boldsymbol  \sigma), {\rm tr}_2(\boldsymbol  \tau))_{\partial K} - ( {\rm tr}_2(\boldsymbol  \sigma), {\rm tr}_1(\boldsymbol  \tau))_{\partial K} \notag\\
   &\quad + \sum_{F\in \mathcal F(K)}\sum_{e\in \partial F}(\boldsymbol  n \cdot\boldsymbol  \sigma \times \boldsymbol  n, \boldsymbol  t_{F,e}\cdot \boldsymbol  \tau)_{e} \notag\\
   &\quad - \sum_{F\in \mathcal F(K)}\sum_{e\in \partial F}(\boldsymbol  t_{F,e}\cdot \boldsymbol  \sigma, \boldsymbol  n \cdot\boldsymbol  \tau \times \boldsymbol  n)_{e}. \label{eq:incsymGreenidentity}
\end{align}
\end{theorem}

 As both $\boldsymbol  \sigma$ and $\boldsymbol  \tau$ are symmetric, by taking transpose of the boundary terms, we can get another equivalent version of Green's identity. For example, the edge term can be $-(\boldsymbol  n\times \boldsymbol  \sigma\cdot \boldsymbol  n, \boldsymbol  \tau \cdot \boldsymbol  t_{F,e})_e$. 

When the domain $\Omega$ is decomposed into a polyhedral mesh, for piecewise smooth function to be in $\boldsymbol  H(\inc,\Omega;\mathbb S)$, the edge terms across different elements should be canceled. 

\begin{remark}\label{eq:edgecancelation}
We note that $\tr_1(\boldsymbol  \tau)$ is independent of the choice of the direction of normal vectors but $\tr_2(\boldsymbol  \tau)$ is an odd function of $\boldsymbol  n$ in the sense that $\tr_2(\boldsymbol  \tau; - \boldsymbol  n) = - \tr_2(\boldsymbol  \tau; \boldsymbol  n)$. Therefore if $\tr_1(\boldsymbol  \tau)$ and $|\tr_2(\boldsymbol  \tau)|$ are single valued on face $F$, the face terms will be canceled out when integrated over a mesh of the domain $\Omega$.

The edge vector $\boldsymbol  t_{F,e}$ in \eqref{eq:incsymGreenidentity} is the orientation of edge $e$ induced by the outwards normal vector $\boldsymbol  n_{\partial K}$ of the face $F$ with respect to $K$. Therefore, for an interior face $F=K\cap K'$, $\boldsymbol  t_{F(K),e} = - \boldsymbol  t_{F(K'),e}$, where $F(K)$ means $F\in \mathcal F(K)$ with normal vector $\boldsymbol  n_{\partial K}$. 

A sufficient condition for the cancelation of edge terms is to impose the continuity of $\boldsymbol  \tau\mid _e$, which implies $\boldsymbol  \tau(\delta)$ is also continuous at vertices. Those observation will be helpful when designing degree of freedoms for finite elements.  
\end{remark}

Most discussion in Remark \ref{eq:edgecancelation} can be found in~\cite{ArnoldAwanouWinther2008} but the Green's identity \eqref{eq:incsymGreenidentity} and the form of $\tr_2(\boldsymbol  \tau)$ cf. \eqref{eq:tr2prop1} seems new. When the domain is smooth, the edge jump will be replaced by a curvature term, cf.~\cite[Theorem 3.16]{Amstutz;Van-Goethem:2016Analysis}, where a different Green's formula on smooth domains is derived. 

\subsection{Trace complexes}\label{sec:tracecomplex}
For a vector $\boldsymbol  v\in \mathbb R^3$, define the tangential trace and the normal trace as
$$
{\rm tr}_1(\boldsymbol  v) := \boldsymbol  v\times \boldsymbol  n, \quad {\rm tr}_2(\boldsymbol  v) := \boldsymbol  v \cdot \boldsymbol  n. 
$$
For a smooth and symmetric tensor $\boldsymbol  \sigma \in \boldsymbol  H(\div, K; \mathbb S)$ define the normal-normal trace and the normal-tangential trace as
$$
{\rm tr}_1(\boldsymbol  \sigma) := \boldsymbol  n\cdot \boldsymbol  \sigma \cdot \boldsymbol  n, \quad {\rm tr}_2(\boldsymbol  \sigma) := \boldsymbol  n\times \boldsymbol  \sigma\cdot  \boldsymbol  n. 
$$

Then we will have the following trace complexes
\begin{equation}\label{eq:tracecomplex1}
\begin{array}{c}
\xymatrix{
\boldsymbol  a\times \boldsymbol  x + \boldsymbol  b \ar[d]^{\tr_1} \ar[r]^-{\subset} & \boldsymbol  v \ar[d]^{\tr_1} \ar[r]^-{\defm}
                & \boldsymbol  \tau \ar[d]^{\tr_1}   \ar[r]^-{\inc} & \ar[d]^{\tr_1}\boldsymbol  \sigma \ar[r]^{\div} & p \\
a_F \boldsymbol  x_F + \boldsymbol  b_F \ar[r]^{\subset} & \boldsymbol  v\times \boldsymbol  n \ar[r]^{\sym \curl_F}
                & \boldsymbol  n\times \boldsymbol  \tau \times \boldsymbol  n   \ar[r]^{\mathrm{div}_F{\mathrm{div}}_F} &  \boldsymbol  n\cdot \boldsymbol  \sigma \cdot \boldsymbol  n \ar[r]^{}& 0    }
\end{array},
\end{equation}
and
\begin{equation}\label{eq:tracecomplex2}
\begin{array}{c}
\xymatrix{
\boldsymbol  a\times \boldsymbol  x + \boldsymbol  b \ar[d]^{\tr_2} \ar[r]^-{\subset} & \boldsymbol  v \ar[d]^{\tr_2} \ar[r]^-{\defm}
                & \boldsymbol  \tau \ar[d]^{\tr_2}   \ar[r]^-{\inc} & \ar[d]^{\tr_2}\boldsymbol  \sigma \ar[r]^{\div} & p \\
\boldsymbol  a_F \cdot \boldsymbol  x_F + b_F \ar[r]^{\subset} & \boldsymbol  v\cdot \boldsymbol  n \ar[r]^{\nabla_F^2}
                & {\rm tr}_2(\boldsymbol  \tau)   \ar[r]^{\nabla_F^{\bot}\cdot} &  \boldsymbol  n\cdot \boldsymbol  \sigma \times \boldsymbol  n \ar[r]^{}& 0    }
\end{array}.
\end{equation}
In \eqref{eq:tracecomplex1} and \eqref{eq:tracecomplex2}, we present the concrete form instead of Sobolev spaces as we will work mostly on polynomial functions which are smooth enough to define the trace point-wisely.

\begin{lemma}
For any sufficiently smooth vector function $\boldsymbol  v$, we have
\begin{align}
\label{eq:tr2complex1}
\boldsymbol  n\times (\defm\boldsymbol  v )\times\boldsymbol  n &= \sym \curl_F(\boldsymbol  v\times\boldsymbol  n),\\
\label{eq:tr2complex2}
{\rm tr}_2(\defm\boldsymbol  v) &= \nabla_F^2(\boldsymbol  v\cdot\boldsymbol  n).
\end{align}
\end{lemma}
\begin{proof}
Using our notation, the first identity \eqref{eq:tr2complex1} is straightforward:
$$
\boldsymbol  n\times (\nabla \boldsymbol  v)\times \boldsymbol  n = \nabla_F^{\bot} \boldsymbol  v \times \boldsymbol  n = (\curl_F \boldsymbol  v\times \boldsymbol  n)^{\intercal}. 
$$
Then apply the $\sym$ operator to get \eqref{eq:tr2complex1}. 

Let $\bs\tau=\defm\boldsymbol  v$. Using \eqref{eq:columncurl} and $\nabla \times \nabla \boldsymbol  v = 0$, it follows that
\begin{align*}
{\rm tr}_2(\boldsymbol  \tau) & = \Pi_F(\boldsymbol  \tau\times\nabla)\times\boldsymbol  n + \nabla_F(\boldsymbol  n\cdot \boldsymbol  \tau \ \Pi_F)\\
& = \frac{1}{2}\Pi_F(\nabla \boldsymbol  v\times\nabla)\times\boldsymbol  n + \frac{1}{2}\nabla_F(\partial_n(\Pi_F\boldsymbol  v)+\nabla_F(\boldsymbol  v\cdot\boldsymbol  n)) \\
& = \frac{1}{2}\nabla_F\big( \nabla_F (\boldsymbol  v\cdot \boldsymbol  n) - \partial_n \Pi_F\boldsymbol  v  \big)+ \frac{1}{2}\nabla_F(\partial_n(\Pi_F\boldsymbol  v)+\nabla_F(\boldsymbol  v\cdot\boldsymbol  n)) \\
&=\nabla_F^2(\boldsymbol  v\cdot\boldsymbol  n),
\end{align*}
as required.
\end{proof}

We then verify the second block.
\begin{lemma}
For any sufficiently smooth and symmetric tensor $\bs\tau$, it holds that
\begin{align}\label{eq:inctr1}
\boldsymbol  n\cdot(\nabla\times\bs\tau\times\nabla) \cdot \boldsymbol  n &= \div_F\div_F(\boldsymbol  n\times\bs\tau\times\boldsymbol  n),\\
\label{eq:inctr2}
\boldsymbol  n \cdot(\nabla\times\bs\tau\times\nabla) \times \boldsymbol  n  & = \nabla_F^{\bot}\cdot  {\rm tr}_2(\boldsymbol  \tau).
\end{align}
\end{lemma}
\begin{proof}
The first identity is from direct computation
$$
\boldsymbol  n\cdot(\nabla\times\bs\tau\times\nabla) \cdot \boldsymbol  n = \nabla_F^{\bot} \cdot \, \Pi_F \,\boldsymbol  \tau \, \Pi_F \cdot \nabla_F^{\bot} = \div_F\div_F(\boldsymbol  n\times\bs\tau\times\boldsymbol  n).
$$
To prove the second identity, we use the trace representation form \eqref{eq:tr2prop1} and the fact $\nabla_F^{\bot}\cdot \nabla_F = 0$ to get 
$$
\boldsymbol  n\cdot (\nabla\times\bs\tau\times\nabla) \times \boldsymbol  n = \nabla_F^{\bot}  \cdot(\bs\tau \times \nabla) \times \boldsymbol  n = \nabla_F^{\bot} \cdot {\rm tr}_2(\boldsymbol  \tau).
$$
\end{proof}

\subsection{Continuity on edges}
In order to construct an $H(\inc)$-conforming finite element, the trace complex inspires us to adopt $H(\div_F\div_F, F;\mathbb S)$ conforming finite element to discretize $\boldsymbol  n\times\bs\tau\times\boldsymbol  n$, and $\boldsymbol  H(\rot_F, F;\mathbb S)$-conforming finite element to discretize ${\rm tr}_2(\boldsymbol  \tau)$. The trace for $\boldsymbol  v_F\in  \boldsymbol  H(\rot_F, F;\mathbb S)$ is $\boldsymbol  v_F\cdot \boldsymbol  t$ on $\partial F$. Two trace operators for $H(\div_F\div_F, F;\mathbb S)$ are identified in~\cite[Lemma 2.1]{ChenHuang2020} and will be recalled below.

\begin{lemma} [Green's identity for the two dimensional $\div\div$ operator]\label{lm:Green}
Let $F$ be a polygon, and let $\boldsymbol  \tau\in \mathcal C^2(F; \mathbb S)$ and $v\in H^2(F)$. Then we have
\begin{align}
(\div_F\div_F\boldsymbol \tau, v)_K&=(\boldsymbol \tau, \nabla^2_F v)_K -\sum_{e\in\mathcal E(F)}\sum_{\delta\in\partial e}\sign_{e,\delta}(\boldsymbol  t_{F,e}\cdot \boldsymbol \tau \cdot \boldsymbol  n_{F,e})(\delta)v(\delta) \notag\\
&\quad - \sum_{e\in\mathcal E(F)}\big[(\boldsymbol  n_{F,e}\cdot \boldsymbol \tau\cdot \boldsymbol  n_{F,e}, \partial_{n_{F,e}} v)_{e} \notag \\
&\qquad \qquad \quad -(\partial_{t}(\boldsymbol  t_{F,e}\cdot \boldsymbol \tau \cdot \boldsymbol  n_{F,e})+\boldsymbol  n_{F,e}\cdot \div_{F}\boldsymbol \tau,  v)_{e}\big], \label{eq:greenidentitydivdiv}
\end{align}
where
\[
\sign_{e,\delta}:=\begin{cases}
1, & \textrm{ if } \delta \textrm{ is the end point of } e, \\
-1, & \textrm{ if } \delta \textrm{ is the start point of } e.
\end{cases}
\]
\end{lemma}
Based on Green's identity \eqref{eq:greenidentitydivdiv}, two traces for $H(\div_F\div_F, F;\mathbb S)$ function are
\begin{equation*}
\boldsymbol  n_{F,e}\cdot \boldsymbol \tau\cdot \boldsymbol  n_{F,e}, \quad \partial_{t}(\boldsymbol  t_{F,e}\cdot \boldsymbol \tau \cdot \boldsymbol  n_{F,e})+\boldsymbol  n_{F,e}\cdot \div_{F}\boldsymbol \tau.
\end{equation*}



\begin{lemma}
Let $F\in\mathcal F(K)$, $e\in\mathcal E(F)$, $\boldsymbol  t$ is a direction vector of $e$, and $\boldsymbol  n_{F,e}:=\boldsymbol  t \times\boldsymbol  n$.
For any sufficiently smooth and symmetric tensor $\bs\tau$, we have on edge $e$ that
\begin{align}
\label{eq:edgetrtt}
\boldsymbol  n_{F,e}\cdot {\rm tr}_1(\boldsymbol  \tau)\cdot \boldsymbol  n_{F,e}& = -\boldsymbol  t\cdot \boldsymbol  \tau \cdot \boldsymbol  t,\\
\label{eq:edgetrprop1}
\partial_t(\boldsymbol  t\cdot{\rm tr}_1(\boldsymbol  \tau)\cdot\boldsymbol  n_{F,e}) + \boldsymbol  n_{F,e}\cdot\div_F({\rm tr}_1(\boldsymbol  \tau)) &= \partial_t(\boldsymbol  n_{F,e}\cdot\bs\tau\cdot\boldsymbol  t) - \boldsymbol  n\cdot(\nabla\times\bs\tau)\cdot\boldsymbol  t,\\
\label{eq:edgetrprop2}
 {\rm tr}_2(\boldsymbol  \tau)\cdot \boldsymbol  t &= \boldsymbol  n \times (\nabla \times \boldsymbol  \tau)\cdot \boldsymbol  t + \partial_t(\Pi_F \boldsymbol  \tau \cdot \boldsymbol  n).
\end{align}
\end{lemma}
\begin{proof}
Let us compute 
\begin{equation}\label{eq:tn}
(\boldsymbol  \tau \times \boldsymbol  n)\cdot \boldsymbol  n_{F,e} = ((\boldsymbol  \tau \times \boldsymbol  n) \times \boldsymbol  n)\cdot (\boldsymbol  n_{F,e}\times \boldsymbol  n) = \boldsymbol  \tau\cdot \boldsymbol  t. 
\end{equation}
Then identity \eqref{eq:edgetrtt} follows. The identity \eqref{eq:edgetrprop1} follows from
$$
\partial_t(\boldsymbol  t\cdot(\boldsymbol  n\times\bs\tau\times\boldsymbol  n)\cdot\boldsymbol  n_{F,e}) = \partial_t(\boldsymbol  n_{F,e}\cdot\bs\tau\cdot\boldsymbol  t),
$$
and using \eqref{eq:rotF} and \eqref{eq:tn}
\begin{align*}
&\boldsymbol  n_{F,e}\cdot\div_F(\boldsymbol  n\times\bs\tau\times\boldsymbol  n)  =  \nabla_F\cdot (\boldsymbol  n\times\bs\tau\times\boldsymbol  n)\cdot \boldsymbol  n_{F,e}\\
 = &- \boldsymbol  n \cdot \nabla \times ((\boldsymbol  \tau \times\boldsymbol  n)\cdot \boldsymbol  n_{F,e}) = - \boldsymbol  n \cdot (\nabla \times \boldsymbol  \tau)\cdot \boldsymbol  t. 
\end{align*}
The identity \eqref{eq:edgetrprop2} is a direct consequence of \eqref{eq:tr2prop1}.
\end{proof}

Those formulae on the edge trace suggests the continuity of $\boldsymbol  \tau \cdot \boldsymbol  t$ and $(\nabla \times \boldsymbol  \tau)\cdot \boldsymbol  t$ on edges. As we mentioned before, in view of Green's identity \eqref{eq:incsymGreenidentity}, it is sufficient to impose the whole tensor $\boldsymbol  \tau$ is continuous on edges. The continuity of $(\nabla \times \boldsymbol  \tau)\cdot \boldsymbol  t$ is not surprising as $\nabla \times \boldsymbol  \tau \times \nabla \in \boldsymbol  H(\div;\mathbb S)$ and thus the normal trace $((\nabla \times \boldsymbol  \tau) \times \nabla) \cdot \boldsymbol  n = (\nabla \times \boldsymbol  \tau)\cdot \nabla_F^{\bot}\in \bs L^2(F;\mathbb R^2)$. Namely $\nabla \times \boldsymbol  \tau\in \boldsymbol  H(\rot_F,F;\mathbb R^2)$ and the edge trace of $\boldsymbol  H(\rot_F,F;\mathbb R^2)$ implies the continuity of $(\nabla \times \boldsymbol  \tau)\cdot \boldsymbol  t$. 

\subsection{Bubble complexes}
We give characterization of bubble functions following~\cite{ArnoldAwanouWinther2008}. Let $K$ be a tetrahedron with vertices $\boldsymbol  x_1$, $\boldsymbol  x_2$, $\boldsymbol  x_3$ and $\boldsymbol  x_4$.
We label the face opposite to $\boldsymbol  x_i$ as the $i$-th face $F_i$,
and denote by $\boldsymbol  n_i$ the unit outwards normal vector of face $F_i$.
Set
$\boldsymbol{N}_{i,j}:=\sym(\boldsymbol  n_k \boldsymbol  n_l^{\intercal}) = \frac{1}{2}(\boldsymbol{n}_k\boldsymbol{n}_l^{\intercal}+\boldsymbol{n}_l\boldsymbol{n}_k^{\intercal})$, where $(ijkl)$ is a permutation cycle of $(1234)$.
Then it is shown in  
\cite{ChenHuHuang2018} that the $6$ symmetric tensors $\{\boldsymbol{N}_{i,j}, i,j=1,2,3,4, i<j\}$ form a basis of $\mathbb{S}$.

Define a tangential-tangential bubble function space of tensorial polynomials of degree $k$ as
\[
\boldsymbol{B}_{K,k}^t:= \mathbb{P}_{k}(K; \mathbb{S})\cap \ker({\rm tr}_1).
\]
It is easy to verify $\tr_1(\lambda_i\lambda_j \boldsymbol{N}_{i,j}) = 0$. 
Since the dimension of $\boldsymbol{B}_{K,k}^t$ is $k(k^2-1)$ (cf.~\cite[Lemma~6.1]{ArnoldAwanouWinther2008}), we have
$$
\boldsymbol{B}_{K,k}^t=\mathbb P_{k-2}(K) \otimes \{ \lambda_i\lambda_j \boldsymbol{N}_{i,j} \} = \sum_{1\leq i<j\leq 4}\mathbb P_{k-2}(K)\lambda_i\lambda_j \boldsymbol{N}_{i,j}.
$$
Define an $\boldsymbol{H}(\inc, K; \mathbb{S})$ bubble function space of polynomials of degree $k$ as
\[
\boldsymbol{B}_{K,k}:= \mathbb{P}_{k}(K; \mathbb{S})\cap \ker({\rm tr}_1) \cap \ker({\rm tr}_2) = \boldsymbol{B}_{K,k}^t \cap \ker({\rm tr}_2).
\]
According to Lemma~6.2 in~\cite{ArnoldAwanouWinther2008}, for any $\bs\tau\in \boldsymbol{B}_{K,k}$, it holds $\bs\tau\mid _e=\bs0\quad\forall~e\in\mathcal E(K).$
Thus
$$
\bs\tau\in \sum_{1\leq i<j\leq 4}\lambda_i\lambda_j\big(\lambda_k \mathbb P_{k-3}(K)+\lambda_l \mathbb P_{k-3}(K)\big)\boldsymbol{N}_{i,j}.
$$
Although there is no precise characterization of $\boldsymbol{B}_{K,k}$, it is shown in~\cite{ArnoldAwanouWinther2008} that the dimension of $\boldsymbol{B}_{K,k}$ is $k^3 - 6k^2 + 11k$.

Furthermore the bubble polynomials for the elasticity complex with $k\geq4$  is established in~\cite[Lemma~7.1]{ArnoldAwanouWinther2008} and~\cite[Lemma~3.2]{HuZhang2015}
\begin{equation}\label{eq:elascomplexbubble}
\resizebox{.92\hsize}{!}{$
\bs0\autorightarrow{$\subset$}{} b_K\mathbb P_{k-3}(K;\mathbb R^3)\autorightarrow{$\defm$}{}\boldsymbol{B}_{K,k}\autorightarrow{$\inc$}{} \boldsymbol{B}_{K,k-2}^n \autorightarrow{$\boldsymbol{\div}$}{} \mathbb P_{k-3}(K;\mathbb R^3)/\bs{RM}\autorightarrow{}{}\bs0,
$}
\end{equation}
where $b_K = \lambda_1\lambda_2\lambda_3\lambda_4$ is the volume bubble polynomial and $\boldsymbol  B_{K,k}^n = \mathbb P_{k}(K;\mathbb S)\cap \boldsymbol  H_0(\div, K;\mathbb S)$ is the $\boldsymbol  H(\div; \mathbb S)$ bubble function space and is characterized in~\cite{HuZhang2015}
\begin{equation}\label{eq:divbubble}
\boldsymbol  B_{K,k}^n=\sum_{0\leq i<j\leq 3}\lambda_i\lambda_j \mathbb P_{k-2}(K)\boldsymbol{T}_{i,j}, \quad k\geq 2
\end{equation}
with $\boldsymbol  T_{i,j} := \boldsymbol{t}_{i,j}\boldsymbol{t}_{i,j}^{\intercal}$ and $\boldsymbol{t}_{i,j}:=\boldsymbol x_j-\boldsymbol x_i$. 

Similarly we also have two dimensional bubble complexes on face $F$. The bubble function space $\mathbb P_{k}(F; \mathbb S)\cap \boldsymbol  H_0(\div_F\div_F, F;\mathbb S)$ is 
\begin{align*}
\{ \boldsymbol  \tau \in \mathbb P_{k}(F; \mathbb S):  \,&\partial_{t}(\boldsymbol  t_{F,e}\cdot \boldsymbol \tau \cdot \boldsymbol  n_{F,e})+\boldsymbol  n_{F,e}\cdot \div_{F}\boldsymbol \tau = 0, \\
&\boldsymbol  n_{F,e}\cdot \boldsymbol \tau\cdot \boldsymbol  n_{F,e} = 0\;\; \forall~e\in \mathcal E(F),\;\; \boldsymbol \tau(\delta)=\boldsymbol 0\;\;\forall~\delta\in \mathcal V(F)\}.
\end{align*}
 We present the results below and a proof of \eqref{eq:divdivbubble} can be found in \cite{ChenHuang2020}. 
\begin{equation}\label{eq:divdivbubble}
\resizebox{.92\hsize}{!}{$
\bs0\autorightarrow{$\subset$}{} b_F\mathbb P_{k-2}(F;\mathbb R^2)\autorightarrow{$\sym\curl_F$}{} \mathbb P_{k}(F; \mathbb S)\cap \boldsymbol  H_0(\div_F\div_F, F;\mathbb S)\autorightarrow{$\div_F\div_F$}{} \mathbb P_{k-2}(F)/\mathbb P_1(F)\autorightarrow{}{}\bs0,
$}
\end{equation}
For the two dimensional Hessian polynomial complex, we have
\begin{equation}\label{eq:hessianbubble}
\resizebox{.92\hsize}{!}{$
\bs0\autorightarrow{$\subset$}{} b_F^2\mathbb P_{k-5}(F)\autorightarrow{$\nabla_F^2$}{} \mathbb P_{k-1}(F; \mathbb S)\cap \boldsymbol  H_0(\rot_F, F;\mathbb S)\autorightarrow{$\rot_F$}{} \mathbb P_{k-2}(F;\mathbb R^2)/\bs{RM}^{\bot}\autorightarrow{}{}\bs0,
$}
\end{equation}
which is a rotation of the 2D elasticity bubble complex established in~\cite{Arnold;Winther:2002finite}.

\section{Finite Element Elasticity Complex}
In this section we present a finite element elasticity complex. In the complex, the $H^1$ conforming finite element is the smooth finite element developed by Neilan for the Stokes complex~\cite{Neilan2015}. The $\boldsymbol  H(\div;\mathbb S)$-conforming finite element is the Hu-Zhang element~\cite{Hu2015,HuZhang2015}. The missing component is $\boldsymbol  H(\inc; \mathbb S)$-conforming finite element which is the focus of this section. 
 
\subsection{$H^1$ conforming finite element for vectors}
Recall the $H^1$ conforming finite element for vectors by Neilan in~\cite{Neilan2015}.
The space of shape functions is chosen as $\mathbb P_{k+1}(K;\mathbb R^3)$ for $k+1\geq 7$.
The degrees of freedom are
\begin{align}
\boldsymbol v (\delta), \nabla \boldsymbol v (\delta), \nabla^2\boldsymbol v (\delta) & \quad\forall~\delta\in \mathcal V(K), \label{H1femdof1}\\
(\boldsymbol v, \boldsymbol  q)_e & \quad\forall~\boldsymbol  q\in\mathbb P_{k-5}(e;\mathbb R^3),  e\in\mathcal E(K),\label{H1femdof2}\\
(\partial_{n_i}\boldsymbol v, \boldsymbol  q)_e & \quad\forall~\boldsymbol  q\in\mathbb P_{k-4}(e;\mathbb R^3),  e\in\mathcal E(K), i=1,2,\label{H1femdof3}\\
(\boldsymbol v, \boldsymbol  q)_F & \quad\forall~\boldsymbol  q\in\mathbb P_{k-5}(F;\mathbb R^3),  F\in\mathcal F(K),\label{H1femdof4}\\
(\boldsymbol v, \boldsymbol q)_K & \quad\forall~\boldsymbol q\in \mathbb P_{k-3}(K;\mathbb R^3). \label{H1femdof5} 
\end{align}
The Neilan element has extra smoothness at vertices and edges. 
Note that the normal derivative is only continuous on edges not on faces and thus this element is only in $H^1$ not $H^2$. To construct an $H^2$-conforming element on tetrahedron, the degree of polynomial will be higher, i.e. $k+1\geq 9$; see Zhang~\cite{Zhang:2009family}. 

\subsection{$H(\div)$ conforming finite element for symmetric tensors}
Recall the $\boldsymbol  H(\div)$ conforming finite element for symmetric tensors in~\cite{HuZhang2015}.
The space of shape functions is chosen as $\mathbb P_{k-2}(K;\mathbb S)$ for $k-2\geq 4$.
The degrees of freedom are
\begin{align}
\boldsymbol\tau(\delta) & \quad\forall~\delta\in \mathcal V(K), \label{HdivSfemdof1}\\
(\boldsymbol  n_i^{\intercal}\boldsymbol\tau\boldsymbol  n_j, q)_e & \quad\forall~\boldsymbol  q\in\mathbb P_{k-4}(e),  e\in\mathcal E(K), i,j=1,2,\label{HdivSfemdof2}\\
(\boldsymbol  n_i^{\intercal}\boldsymbol\tau\boldsymbol  t, q)_e & \quad\forall~\boldsymbol  q\in\mathbb P_{k-4}(e),  e\in\mathcal E(K), i=1,2,\label{HdivSfemdof3}\\
(\boldsymbol\tau\boldsymbol n, \boldsymbol  q)_F & \quad\forall~\boldsymbol  q\in\mathbb P_{k-5}(F;\mathbb R^3),  F\in\mathcal F(K),\label{HdivSfemdof4}\\
(\boldsymbol\tau, \boldsymbol q)_K & \quad\forall~\boldsymbol q\in \mathbb P_{k-4}(K;\mathbb S). \label{HdivSfemdof5} 
\end{align}
The unisovlence can be proved as follows. The boundary degree of freedom \eqref{HdivSfemdof1}-\eqref{HdivSfemdof4} will determine the trace $\boldsymbol  \tau \boldsymbol  n$ uniquely by the unisolvence of the Lagrange elements. The interior part will be determined by \eqref{HdivSfemdof5} due to the characterization of $\boldsymbol  H(\div, K; \mathbb S)$ bubble function, cf. \eqref{eq:divbubble}.

\subsection{$H(\inc)$ conforming finite element for symmetric tensors}
With previous preparations, we can construct an $H(\inc)$ conforming finite element now.
Take $\mathbb P_k(K;\mathbb S), k \geq 6,$ as the space of shape functions. The degrees of freedom are
\begin{align}
\boldsymbol\tau (\delta), \nabla\boldsymbol\tau (\delta) & \quad\forall~\delta\in \mathcal V(K), \label{Hincfem3ddof1}\\
(\nabla\times\boldsymbol\tau\times\nabla) (\delta) & \quad\forall~\delta\in \mathcal V(K), \label{Hincfem3ddof1-1}\\
(\boldsymbol\tau, \boldsymbol  q)_e & \quad\forall~\boldsymbol  q\in\mathbb P_{k-4}(e;\mathbb S),  e\in\mathcal E(K),\label{Hincfem3ddof2}\\
(\nabla\times\boldsymbol\tau\cdot\boldsymbol  t, \boldsymbol  q)_e & \quad\forall~\boldsymbol  q\in\mathbb P_{k-3}(e;\mathbb R^3),  e\in\mathcal E(K),\label{Hincfem3ddof3}\\
(\boldsymbol  n_i^{\intercal}(\nabla\times\boldsymbol\tau\times\nabla)\boldsymbol  n_j, q)_e & \quad\forall~q\in\mathbb P_{k-4}(e),  e\in\mathcal E(K), i,j=1,2,\label{Hincfem3ddof4}\\
(\boldsymbol  n_i^{\intercal}(\nabla\times\boldsymbol\tau\times\nabla)\boldsymbol  t, q)_e & \quad\forall~q\in\mathbb P_{k-4}(e),  e\in\mathcal E(K), i=1,2,\label{Hincfem3ddof5}\\
(\boldsymbol  n\times\boldsymbol\tau\times\boldsymbol  n, \boldsymbol  q)_F & \quad\forall~\boldsymbol  q\in\nabla_F^2\mathbb P_{k-5}(F),  F\in\mathcal F(K),\label{Hincfem3ddof6}\\
(\boldsymbol  n\times\boldsymbol\tau\times\boldsymbol  n, \boldsymbol  q)_F & \quad\forall~\boldsymbol  q\in\sym(\boldsymbol  x^{\perp}\mathbb P_{k-5}(F;\mathbb R^2)),  F\in\mathcal F(K),\label{Hincfem3ddof7}\\
({\rm tr}_2(\boldsymbol  \tau), \boldsymbol  q)_F & \quad\forall~\boldsymbol  q\in\sym \curl_F\mathbb P_{k-5}(F;\mathbb R^2),  F\in\mathcal F(K),\label{Hincfem3ddof8}\\
({\rm tr}_2(\boldsymbol  \tau), \boldsymbol  q)_F & \quad\forall~\boldsymbol  q\in\boldsymbol  x\boldsymbol  x^{\intercal}\mathbb P_{k-5}(F),  F\in\mathcal F(K),\label{Hincfem3ddof9}\\
(\boldsymbol\tau, \boldsymbol q)_K & \quad\forall~\boldsymbol q\in \inc\mathbb P_{k-4}(K;\mathbb S)\oplus\sym(\mathbb P_{k-3}(K;\mathbb R^3)\boldsymbol  x^{\intercal}). \label{Hincfem3ddof10}
\end{align}

We first show the trace is uniquely determined by the degree of freedom \eqref{Hincfem3ddof1}-\eqref{Hincfem3ddof9} on the boundary.  
\begin{lemma}\label{lem:unisovlenHincfempre0}
Let $F\in\mathcal F(K)$ and $\boldsymbol  \tau \in \mathbb P_k(K;\mathbb S)$.
If all the degrees of freedom \eqref{Hincfem3ddof1}-\eqref{Hincfem3ddof9} on face $F$ vanish, then $\tr_1(\boldsymbol  \tau) =\boldsymbol 0$ and ${\rm tr}_2(\boldsymbol  \tau)=\boldsymbol  0$ on face $F$.
\end{lemma}
\begin{proof}
We split our proof into several steps. For the easy of notation, denote by $\boldsymbol  \sigma =  \nabla\times\bs\tau\times\nabla \in \mathbb P_{k-2}(K;\mathbb S)$. 

\medskip
\noindent {\em Step 1. Traces on edges are vanished.}
By the vanishing degrees of freedom \eqref{Hincfem3ddof1}, \eqref{Hincfem3ddof2}, and \eqref{Hincfem3ddof3}, $\bs\tau\mid _e=\bs0$ and $(\nabla\times\boldsymbol\tau\cdot\boldsymbol  t)\mid _e=\bs0$ for any edge $e\in\mathcal E(F)$.
Then it follows from \eqref{eq:edgetrprop1} and \eqref{eq:edgetrprop2} that 
\begin{align}
\label{eq:20210117-0}
\boldsymbol  n_{F,e}\cdot {\rm tr}_1(\bs\tau)\cdot \boldsymbol  n_{F,e}&= 0,\\
\label{eq:20210117-1}
\big(\partial_t(\boldsymbol  t\cdot{\rm tr}_1(\bs\tau)\cdot \boldsymbol  n_{F,e}) + \boldsymbol  n_{F,e}\cdot\div_F{\rm tr}_1(\bs\tau)\big) &= 0, \\
\label{eq:20210117-2}
 {\rm tr}_2(\boldsymbol  \tau) \cdot \boldsymbol  t_e &= \bs0.
\end{align}
By the vanishing degree of freedom \eqref{Hincfem3ddof1-1}, \eqref{Hincfem3ddof4}, and \eqref{Hincfem3ddof5} for $\boldsymbol  \sigma$, we know all components of $\boldsymbol  \sigma\mid _e$, except $\boldsymbol  t\cdot \boldsymbol  \sigma \cdot \boldsymbol  t$, are zero. 

\medskip

\noindent{\em Step 2. ${\rm tr}_1(\boldsymbol  \tau)$ is vanished.}
Applying the Green's identity \eqref{eq:greenidentitydivdiv} for the $\div_F\div_F$ operator, we get from \eqref{eq:inctr1}, \eqref{eq:20210117-0}-\eqref{eq:20210117-1} and \eqref{Hincfem3ddof6} that
$$
(\boldsymbol  n\cdot \boldsymbol  \sigma \cdot \boldsymbol  n, q)_F = (\div_F\div_F {\rm tr}_1(\boldsymbol  \tau), q)_F=({\rm tr}_1(\boldsymbol  \tau), \nabla_F^2q)_F=0\quad\forall~q\in\mathbb P_{k-5}(F).
$$ 
As $\boldsymbol  n\cdot \boldsymbol  \sigma \cdot \boldsymbol  n\mid _e = 0$, we conclude $\boldsymbol  n\cdot \boldsymbol  \sigma \cdot \boldsymbol  n\mid_F = \div_F\div_F {\rm tr}_1(\boldsymbol  \tau) = 0$. 

Since $\div_F\div_F( {\rm tr}_1(\boldsymbol  \tau) )=0$ and edge traces of ${\rm tr}_1(\boldsymbol  \tau)$ is also vanished, from the bubble complex \eqref{eq:divdivbubble}, there exists $\boldsymbol  v_1\in \mathbb P_{k-2}(F;\mathbb R^2)$ s.t. ${\rm tr}_1(\boldsymbol  \tau) =\sym \curl_F(b_F\boldsymbol  v_1)$. 
Noting that ${\rm tr}_1(\boldsymbol  \tau) |_{\partial F}=\bs0$, we acquire $\boldsymbol  v_1|_{\partial F}=\bs0$. Then there exists $\boldsymbol  v_2\in \mathbb P_{k-5}(F;\mathbb R^2)$ such that ${\rm tr}_1(\boldsymbol  \tau)=\sym \curl_F(b_F^2\boldsymbol  v_2)$.
Due to the vanishing degrees of freedom \eqref{Hincfem3ddof7}, we get from the integration by parts that
$$
(b_F^2\boldsymbol  v_2, \textrm{rot}_F \boldsymbol  q)_F=0\quad\forall~\boldsymbol  q\in \sym(\boldsymbol  x^{\perp}\mathbb P_{k-5}(F;\mathbb R^2)).
$$
Using the fact $\textrm{rot}_F \sym(\boldsymbol  x^{\perp}\mathbb P_{k-5}(F;\mathbb R^2))=\mathbb P_{k-5}(F;\mathbb R^2)$, cf. the complex \eqref{eq:hessiancomplexPolydouble},
we can chose $\bs q$ s.t. $ \textrm{rot}_F\bs q = \bs v_2$ and conclude $\boldsymbol  v_2 = 0$. Therefore $\tr_1(\boldsymbol  \tau) = \boldsymbol  n\times\boldsymbol\tau\times\boldsymbol  n =\boldsymbol 0$.

\medskip

\noindent{\em Step 3. ${\rm tr}_2(\boldsymbol  \tau)$ is vanished.}
Similarly we obtain from \eqref{eq:inctr2}, the integration by parts, \eqref{eq:20210117-2} and \eqref{Hincfem3ddof8} that
$$
(\boldsymbol  n\cdot \boldsymbol  \sigma \times \boldsymbol  n, \boldsymbol  q)_F = (\nabla_F^{\bot}\cdot {\rm tr}_2(\boldsymbol  \tau), \boldsymbol  q)_F=({\rm tr}_2(\boldsymbol  \tau), \sym \curl_F\boldsymbol  q)_F=0\; \forall~\boldsymbol  q\in\mathbb P_{k-5}(F;\mathbb R^2).
$$ 
Together with $\boldsymbol  n\cdot \boldsymbol  \sigma \times \boldsymbol  n \mid_e = 0$, we conclude 
$$
\boldsymbol  n\cdot \boldsymbol  \sigma \times \boldsymbol  n \mid_F = \nabla_F^{\bot}\cdot {\rm tr}_2(\boldsymbol  \tau) \mid_F = 0.
$$
Then by the bubble complex \eqref{eq:hessianbubble}, there exists a $v \in \mathbb P_{k-5}(F)$ such that ${\rm tr}_2(\boldsymbol  \tau) = \nabla_F^2(b_F^2 v)$. Using the fact $\div_F\div_F: \boldsymbol  x \boldsymbol  x^{\intercal}\mathbb P_{k-5}(F) \to \mathbb P_{k-5}(F)$ is bijective, cf. the complex \eqref{eq:divdivcomplexPolydouble}, we can find $\boldsymbol  q\in \boldsymbol  x \boldsymbol  x^{\intercal}\mathbb P_{k-5}(F)$ s.t. $\div_F\div_F \boldsymbol  q = v$. Now using the vanishing degree of freedom \eqref{Hincfem3ddof9}, we get
$$
({\rm tr}_2(\boldsymbol  \tau), \boldsymbol  q)_F = (b_F^2 v, \div_F\div_F \boldsymbol  q)_F = (b_F^2 v, v)_F = 0,
$$
which means $v = 0$ and consequently ${\rm tr}_2(\boldsymbol  \tau) = \boldsymbol0$. 
\end{proof}


Now we are in the position to present the unisolvence. 
\begin{theorem}\label{thm:unisovlenHincfem}
The degrees of freedom \eqref{Hincfem3ddof1}-\eqref{Hincfem3ddof10} are unisolvent for $\mathbb P_{k}(K;\mathbb S)$.
\end{theorem}
\begin{proof}
%
We count the number of degrees of freedom \eqref{Hincfem3ddof1}-\eqref{Hincfem3ddof10} by the dimension of the sub-simplex
\begin{itemize}
 \item $4$ vertices: $4\times (6 + 3\times 6 + 6) = 120;$
 \smallskip
 \item $6$ edges: $6 [ 6(k-3) + 3(k-2) + 3(k-3) + 2(k-3)] = 6[14(k-3)+3]$;
 \smallskip
  \item $4$ faces: $4 \big [ {k-3 \choose 2} - 3 + 2{k-3 \choose 2} +  2 {k-3 \choose 2} - 3 + {k-3 \choose 2} \big ] = 4\big [3(k-3)(k-4)-6\big ];$

  \item $1$ volume: $6{k-1 \choose 3} - 3{k \choose 3} + 6 + 3{k \choose 3} = (k-1)(k-2)(k-3) + 6 = k^3-6k^2+11k.$ 
\end{itemize}
The total dimension is $k^3+6k^2+11k+6$,
which agrees with $\dim\mathbb P_k(K;\mathbb S)$.

Take any $\boldsymbol \tau\in\mathbb P_{k}(K;\mathbb S)$ and
suppose all the degrees of freedom \eqref{Hincfem3ddof1}-\eqref{Hincfem3ddof10} vanish. We are going to prove it is zero. 

First of all, by Lemma~\ref{lem:unisovlenHincfempre0}, we conclude $\tr_1(\boldsymbol  \tau) =\boldsymbol 0$ and ${\rm tr}_2(\boldsymbol  \tau)=\boldsymbol  0$, which immediately induce $(\inc\bs\tau \cdot \boldsymbol  n)|_{\partial K}=\bs0$ from \eqref{eq:inctr1}-\eqref{eq:inctr2}.
Then it follows from \eqref{Hincfem3ddof10} and the Green's identity \eqref{eq:incsymGreenidentity} that
$$
(\inc\bs\tau, \boldsymbol  q)_K=0\quad\forall~\boldsymbol  q\in\mathbb P_{k-4}(K;\mathbb S),
$$
which together with the unisolvence of the Hu-Zhang element (cf. \eqref{HdivSfemdof1}-\eqref{HdivSfemdof5}) implies $\inc\bs\tau=\bs0$. 

By the complex for bubble function spaces \eqref{eq:elascomplexbubble}, there exists $\boldsymbol  v\in\mathbb P_{k-3}(K;\mathbb R^3)$ such that $\bs\tau=\defm(b_K\boldsymbol  v)$. 
Due to the vanishing degrees of freedom \eqref{Hincfem3ddof10}, we get from the integration by parts that
$$
(b_F\boldsymbol  v, \div\boldsymbol  q)_F=0\quad\forall~\boldsymbol  q\in\sym(\mathbb P_{k-3}(K;\mathbb R^3)\boldsymbol  x^{\intercal}).
$$
We finish the proof by using the fact $\div\sym(\boldsymbol  x\mathbb P_{k-3}(K;\mathbb R^3))=\mathbb P_{k-3}(K;\mathbb R^3)$, cf. Lemma \ref{lem:divsymtensoronto} and the complex \eqref{eq:elascomplex3dPolydouble}.
\end{proof}



\subsection{Finite element elasticity complex in three dimensions}

For an integer $k\geq 6$, define global finite elements
\begin{align*}
\boldsymbol  V_h:=\{\boldsymbol  v_h\in \boldsymbol  H^1(\Omega; \mathbb R^3): &\boldsymbol  v_h|_K\in\mathbb P_{k+1}(K;\mathbb R^3) \textrm{ for each } K\in\mathcal T_h, \textrm{ all the } \\
& \textrm{ degrees of freedom \eqref{H1femdof1}-\eqref{H1femdof4} are single-valued} \},\\
\bs\Sigma_h^{\rm inc}:=\{\bs\tau_h\in \boldsymbol  L^2(\Omega; \mathbb S): &\bs\tau_h|_K\in\mathbb P_{k}(K;\mathbb S) \textrm{ for each } K\in\mathcal T_h, \textrm{ all the } \\
& \textrm{ degrees of freedom \eqref{Hincfem3ddof1}-\eqref{Hincfem3ddof9} are single-valued} \},\\
\bs\Sigma_h^{\rm div}:=\{\bs\tau_h\in \boldsymbol  H(\div,\Omega; \mathbb S): &\bs\tau_h|_K\in\mathbb P_{k-2}(K;\mathbb S) \textrm{ for each } K\in\mathcal T_h, \textrm{ all the } \\
& \textrm{ degrees of freedom \eqref{HdivSfemdof1}-\eqref{HdivSfemdof4} are single-valued} \},\\
\mathcal Q_h :=\{\boldsymbol  q_h\in L^2(\Omega;\mathbb R^3): &\boldsymbol  q_h|_K\in \mathbb P_{k-3}(K;\mathbb R^3) \textrm{ for each } K\in\mathcal T_h\}.
\end{align*}
Thanks to Lemma~\ref{lem:unisovlenHincfempre0} and the fact $\boldsymbol  \tau\mid _e$ and $(\nabla\times\boldsymbol\tau\cdot\boldsymbol  t)\mid _e$ are single valued, it holds $\bs\Sigma_h^{\rm inc}\subset \boldsymbol  H(\inc, \Omega; \mathbb S)$, cf. Remark \ref{eq:edgecancelation}. 

Counting the dimensions of these spaces, we have
\begin{align*}
\dim\boldsymbol  V_h&= 30\#\mathcal V_h+ (9k-30)\#\mathcal E_h + \frac{3}{2}(k-3)(k-4)\#\mathcal F_h + \frac{1}{2}k(k-1)(k-2)\#\mathcal T_h,\\
\dim\bs\Sigma_h^{\rm inc}&= 30\#\mathcal V_h+ (14k-39)\#\mathcal E_h + \left(3k^2-21k+30\right)\#\mathcal F_h \\
&\quad+(k^3-6k^2+11k)\#\mathcal T_h,\\
\dim\bs\Sigma_h^{\rm div}&= 6\#\mathcal V_h+ (5k-15)\#\mathcal E_h + \frac{3}{2}(k-3)(k-4)\#\mathcal F_h \\
&\quad+(k-1)(k-2)(k-3)\#\mathcal T_h,\\
\dim\mathcal Q_h &= \frac{1}{2}k(k-1)(k-2)\#\mathcal T_h.
\end{align*}

\begin{lemma}\label{lem:discretekerinc}
Let $\bs\tau\in \bs\Sigma_h^{\rm inc}$ and $\inc\bs\tau=\bs0$.
Then there exists $\boldsymbol  v\in\boldsymbol  V_h$ satisfying $\bs\tau=\defm\boldsymbol  v$.
\end{lemma}
\begin{proof}
By the polynomial elasticity complex \eqref{eq:elascomplex3dPoly} and the elasticity complex \eqref{eq:elasticitycomplex}, there exists $\boldsymbol  v=(v_1, v_2, v_3)^{\intercal}\in \boldsymbol  H^1(\Omega;\mathbb R^3)$ s.t. $\bs\tau=\defm\boldsymbol  v$ and $\boldsymbol  v|_K\in\mathbb P_{k+1}(K;\mathbb R^3)$ for each $K\in\mathcal T_h$. We are going to show such $\boldsymbol  v\in \boldsymbol  V_h$ by verifying the continunity of degree of freedoms \eqref{H1femdof1}-\eqref{H1femdof4}. As an $H^1$ element, $\boldsymbol  v$ is continuous at vertices, edges and faces. The focus is on the derivatives of $\boldsymbol  v$. 

Due to the additional smoothness of $\bs\Sigma_h^{\rm inc}$,
$\nabla(\defm\boldsymbol  v)(\delta)$ is single-valued at each vertex $\delta\in\mathcal V_h$, and 
$(\defm\boldsymbol  v)\mid _e$ and $(\nabla\times(\defm\boldsymbol  v)\cdot\boldsymbol  t)\mid _e$ are single-valued on each edge $e\in\mathcal E_h$.
Next we show that $\nabla\boldsymbol  v$ is single-valued on each edge $e\in\mathcal E_h$. By \eqref{eq:du},
$$
\partial_t\boldsymbol  v=(\nabla \boldsymbol  v)^{\intercal}\cdot\boldsymbol  t = (\defm\boldsymbol  v)\cdot\boldsymbol  t+\frac{1}{2}\mskw(\nabla\times\boldsymbol  v)\cdot\boldsymbol  t
=(\defm\boldsymbol  v)\cdot\boldsymbol  t+\frac{1}{2}(\nabla\times\boldsymbol  v)\times \boldsymbol  t,
$$
\begin{equation}\label{eq:tr2prope}
\nabla\times(\defm\boldsymbol  v)\cdot\boldsymbol  t =\frac{1}{2}\partial_t(\nabla\times\boldsymbol  v)
\end{equation}
on edge $e$ with the unit tangential vector $\boldsymbol  t$. Hence $(\nabla\times\boldsymbol  v)\times \boldsymbol  t$ and $\partial_t(\nabla\times\boldsymbol  v)$ are single-valued across $e$. Take any face $F\in\mathcal F_h^i$ shared by $K_1, K_2\in\mathcal T_h$, it follows from the single-valued $(\nabla\times\boldsymbol  v)\times \boldsymbol  t|_{\partial F}$ that $(\nabla\times\boldsymbol  v)|_{K_1}$ and $(\nabla\times\boldsymbol  v)|_{K_2}$ coincides with each other at the three vertices of $F$. Thus $(\nabla\times\boldsymbol  v)(\delta)$ is single-valued at each vertex $\delta\in\mathcal V_h$, which together with the single-valued $\partial_t(\nabla\times\boldsymbol  v)$ on $\mathcal E_h$ implies that $\nabla\times\boldsymbol  v$ is single-valued on each edge $e\in\mathcal E_h$.
Since $(\nabla \boldsymbol  v)^{\intercal}=\defm\boldsymbol  v+\frac{1}{2}\mskw(\nabla\times\boldsymbol  v)$, $\nabla\boldsymbol  v$ is single-valued on each edge $e\in\mathcal E_h$.

By the identity 
$$
\partial_{ij}v_k = \partial_{i}((\defm\boldsymbol  v)_{jk}) + \partial_{j}((\defm\boldsymbol  v)_{ki}) -\partial_{k}((\defm\boldsymbol  v)_{ij})\quad\textrm{ for } i,j,k=1,2,3,
$$
the tensor $\nabla^2\boldsymbol  v(\delta)$ is single-valued at each vertex $\delta\in\mathcal V_h$ as $\nabla(\defm\boldsymbol  v)(\delta)$ is single-valued.
Therefore $\boldsymbol  v\in\boldsymbol  V_h$.
\end{proof}

\begin{theorem}
The finite element elasticity complex
\begin{equation}\label{eq:elascomplex3dfem}
\boldsymbol{RM}\autorightarrow{$\subset$}{} \boldsymbol  V_h\autorightarrow{$\defm$}{}\bs\Sigma_h^{\rm inc}\autorightarrow{$\inc$}{} \bs\Sigma_h^{\rm div} \autorightarrow{${\div}$}{} \mathcal Q_h\autorightarrow{}{}0
\end{equation}
is exact.
\end{theorem}
\begin{proof}
The inclusion $\defm\boldsymbol  V_h\subseteq \bs\Sigma_h^{\rm inc}$ follows from \eqref{eq:tr2prope} and \eqref{eq:tr2complex1}-\eqref{eq:tr2complex2}, and $\inc\bs\Sigma_h^{\rm inc}\subseteq \bs\Sigma_h^{\div}$ holds from \eqref{eq:inctr1}-\eqref{eq:inctr2} and Lemma~\ref{lem:unisovlenHincfempre0}. The proof of $\div\bs\Sigma_h^{\rm \div}=\mathcal Q_h$ can be found in~\cite{Hu2015,HuZhang2015}. 
Hence \eqref{eq:elascomplex3dfem} is a complex. And
\begin{align*}
\dim(\bs\Sigma_h^{\rm \div}\cap\ker(\div))&=\dim\bs\Sigma_h^{\rm \div}-\dim\mathcal Q_h 
 \\
&=6\#\mathcal V_h+ (5k-15)\#\mathcal E_h + \frac{3}{2}(k-3)(k-4)\#\mathcal F_h \\
&\quad+\frac{1}{2}(k-1)(k-2)(k-6)\#\mathcal T_h.
\end{align*}
It follows from Lemma~\ref{lem:discretekerinc} that $\defm\boldsymbol  V_h=\bs\Sigma_h^{\rm inc}\cap\ker(\inc)$.
Thus
\begin{align*}
\dim(\inc\bs\Sigma_h^{\inc})&=\dim\bs\Sigma_h^{\inc}-\dim\defm\boldsymbol  V_h =\dim\bs\Sigma_h^{\inc}-\dim\boldsymbol  V_h+6
 \\
&=(5k-9)\#\mathcal E_h + \frac{3}{2}(k^2-7k+8)\#\mathcal F_h \\
&\quad+\frac{1}{2}(k^3-9k^2+20k)\#\mathcal T_h+6.
\end{align*}
Then we get from the Euler's identity that
\begin{align*}
\dim(\bs\Sigma_h^{\rm \div}\cap\ker(\div))-\dim(\inc\bs\Sigma_h^{\inc})&=6\#\mathcal V_h-6\#\mathcal E_h+6\#\mathcal F_h-6\#\mathcal T_h+6=0.
\end{align*}
Therefore $\bs\Sigma_h^{\rm \div}\cap\ker(\div)=\inc\bs\Sigma_h^{\inc}$.
\end{proof}

\bibliographystyle{abbrv}
\bibliography{./paper}

\begin{thebibliography}{10}

\bibitem{Amstutz;Van-Goethem:2016Analysis}
S.~Amstutz and N.~Van~Goethem.
\newblock Analysis of the incompatibility operator and application in intrinsic
  elasticity with dislocations.
\newblock {\em SIAM Journal on Mathematical Analysis}, 48(1):320--348, 2016.

\bibitem{Amstutz;Van-Goethem:2019incompatibility}
S.~Amstutz and N.~Van~Goethem.
\newblock The incompatibility operator: from {R}iemann's intrinsic view of
  geometry to a new model of elasto-plasticity.
\newblock In {\em Topics in Applied Analysis and Optimisation}, pages 33--70.
  Springer, 2019.

\bibitem{Arnold;Falk;Winther:2010Finite}
D.~Arnold, R.~Falk, and R.~Winther.
\newblock Finite element exterior calculus: from {H}odge theory to numerical
  stability.
\newblock {\em Bulletin of the American mathematical society}, 47(2):281--354,
  2010.

\bibitem{ArnoldAwanouWinther2008}
D.~N. Arnold, G.~Awanou, and R.~Winther.
\newblock Finite elements for symmetric tensors in three dimensions.
\newblock {\em Math. Comp.}, 77(263):1229--1251, 2008.

\bibitem{Arnold;Falk;Winther:2006Finite}
D.~N. Arnold, R.~S. Falk, and R.~Winther.
\newblock Finite element exterior calculus, homological techniques, and
  applications.
\newblock {\em Acta Numerica}, 15:1, 2006.

\bibitem{Arnold;Hu:2020Complexes}
D.~N. Arnold and K.~Hu.
\newblock Complexes from complexes.
\newblock {\em Found. Comput. Math.}, 2021.

\bibitem{Arnold;Winther:2002finite}
D.~N. Arnold and R.~Winther.
\newblock Mixed finite elements for elasticity.
\newblock {\em Numerische Mathematik}, 92(3):401--419, 2002.

\bibitem{ChenHuHuang2018}
L.~Chen, J.~Hu, and X.~Huang.
\newblock Fast auxiliary space preconditioners for linear elasticity in mixed
  form.
\newblock {\em Math. Comp.}, 87(312):1601--1633, 2018.

\bibitem{Chen;Hu;Huang:2018Multigrid}
L.~Chen, J.~Hu, and X.~Huang.
\newblock Multigrid methods for {H}ellan--{H}errmann--{J}ohnson mixed method of
  {K}irchhoff plate bending problems.
\newblock {\em Journal of Scientific Computing}, 76(2):673--696, 2018.

\bibitem{Chen;Hu;Huang;Man:2018Residual-based}
L.~Chen, J.~Hu, X.~Huang, and H.~Man.
\newblock Residual-based a posteriori error estimates for symmetric conforming
  mixed finite elements for linear elasticity problems.
\newblock {\em Science China Mathematics}, 61(6):973--992, 2018.

\bibitem{Chen;Huang:2020Discrete}
L.~Chen and X.~Huang.
\newblock Discrete {H}essian complexes in three dimensions.
\newblock {\em arXiv preprint arXiv:2012.10914}, 2020.

\bibitem{ChenHuang2020}
L.~Chen and X.~Huang.
\newblock Finite elements for divdiv-conforming symmetric tensors.
\newblock {\em https://arxiv.org/abs/2005.01271}, 2020.

\bibitem{Chen;Huang:2020Finite}
L.~Chen and X.~Huang.
\newblock Finite elements for divdiv-conforming symmetric tensors in three
  dimensions.
\newblock {\em arXiv preprint arXiv:2007.12399}, 2020.

\bibitem{Christiansen:2011linearization}
S.~H. Christiansen.
\newblock On the linearization of {R}egge calculus.
\newblock {\em Numerische Mathematik}, 119(4):613--640, 2011.

\bibitem{Christiansen;Gopalakrishnan;Guzman;Hu:2020discrete}
S.~H. Christiansen, J.~Gopalakrishnan, J.~Guzm{\'a}n, and K.~Hu.
\newblock A discrete elasticity complex on three-dimensional {A}lfeld splits.
\newblock {\em arXiv preprint arXiv:2009.07744}, 2020.

\bibitem{Christiansen;Hu;Hu:2018finite}
S.~H. Christiansen, J.~Hu, and K.~Hu.
\newblock Nodal finite element de {R}ham complexes.
\newblock {\em Numerische Mathematik}, 139(2):411--446, 2018.

\bibitem{christiansen2019finite}
S.~H. Christiansen and K.~Hu.
\newblock Finite element systems for vector bundles: elasticity and curvature.
\newblock {\em arXiv preprint arXiv:1906.09128}, 2019.

\bibitem{ChristiansenHuSande2020}
S.~H. Christiansen, K.~Hu, and E.~Sande.
\newblock Poincar\'{e} path integrals for elasticity.
\newblock {\em J. Math. Pures Appl.}, 135:83--102, 2020.

\bibitem{Ciarlet:2010inequality}
P.~G. Ciarlet.
\newblock On {K}orn's inequality.
\newblock {\em Chinese Annals of Mathematics, Series B}, 31(5):607--618, 2010.

\bibitem{Floater;Hu:2020characterization}
M.~S. Floater and K.~Hu.
\newblock A characterization of supersmoothness of multivariate splines.
\newblock {\em Advances in Computational Mathematics}, 46(5):1--15, 2020.

\bibitem{Fu;Guzman;Neilan:2020smooth}
G.~Fu, J.~Guzm{\'a}n, and M.~Neilan.
\newblock Exact smooth piecewise polynomial sequences on {A}lfeld splits.
\newblock {\em Mathematics of Computation}, 89(323):1059--1091, 2020.

\bibitem{Hu2015}
J.~Hu.
\newblock Finite element approximations of symmetric tensors on simplicial
  grids in {$\Bbb R^n$}: the higher order case.
\newblock {\em J. Comput. Math.}, 33(3):283--296, 2015.

\bibitem{HuLiang2021}
J.~Hu and Y.~Liang.
\newblock Conforming discrete {G}radgrad-complexes in three dimensions.
\newblock {\em Math. Comp.}, 90(330):1637--1662, 2021.

\bibitem{HuZhang2015}
J.~Hu and S.~Zhang.
\newblock A family of symmetric mixed finite elements for linear elasticity on
  tetrahedral grids.
\newblock {\em Sci. China Math.}, 58(2):297--307, 2015.

\bibitem{Kelly:Mechanics}
P.~Kelly.
\newblock Mechanics lecture notes: An introduction to solid mechanics.
\newblock Available from
  http://homepages.engineering.auckland.ac.nz/~pkel015/SolidMechanicsBooks/index.html.

\bibitem{Neilan2015}
M.~Neilan.
\newblock Discrete and conforming smooth de {R}ham complexes in three
  dimensions.
\newblock {\em Math. Comp.}, 84(295):2059--2081, 2015.

\bibitem{Zhang:2009family}
S.~Zhang.
\newblock A family of 3d continuously differentiable finite elements on
  tetrahedral grids.
\newblock {\em Applied Numerical Mathematics}, 59(1):219--233, 2009.

\end{thebibliography}
\end{document}